\documentclass[letterpaper,10pt]{amsart}
\input{./preamble.sty}
\title{Dynamics of Projectivized Toric Vector Bundles}

\begin{document}
\author{Javier Gonz\'alez-Anaya}
\address{Department of Mathematics and Computer Science, 
Santa Clara University, 
Santa Clara, CA 95053, USA}
\email{jgonzalezanaya@scu.edu}

\author{Brett Nasserden}
\address{Department of Mathematics, 
McMaster University, 
Hamilton, ON L8S 4K1, Canada}
\email{nasserdb@mcmaster.ca}

\author{Sasha Zotine}
\address{Department of Mathematics, 
McMaster University, 
Hamilton, ON L8S 4K1, Canada}
\email{zotinea@mcmaster.ca}

\keywords{vector bundles, arithmetic dynamics, toric geometry}
\subjclass[2020]{37P55, 14J60, 14M25}
\date{\today}

\maketitle
\begin{abstract}
We study surjective endomorphisms of projective bundles over toric varieties, achieving three main results. First, we provide a structural theorem describing endomorphisms of projectivized split bundles over arbitrary base varieties, which we use to classify all surjective endomorphisms of Hirzebruch surfaces and construct novel families of examples. Second, for non-split equivariant bundles over toric varieties, we prove that the dynamical degree of an endomorphism of the projectivization is controlled by the base morphism; as a consequence, we establish the Kawaguchi--Silverman conjecture for such bundles. Third, using an explicit transition function method, we prove that projectivizations of tangent and cotangent bundles of smooth toric varieties admit no non-automorphic surjective endomorphisms commuting with toric morphisms on the base.
\end{abstract}

\section{Introduction}
One of the challenges in holomorphic and arithmetic dynamics is that, while it is easy to write down dominant rational maps $\theta\colon X\dashrightarrow X$ between normal projective varieties, studying their dynamical properties is difficult. For example, already for $X=\PP^2$, computing the first dynamical degree $\lambda_1(\theta)$ is a difficult task, and it is known that it may even be a transcendental number~\cite{belldillerjonssonkrieger2024transcendental}. Similarly, from the arithmetic perspective, the theory of canonical heights for dominant rational maps is not well understood. 

A common way to get around these difficulties is by working with surjective endomorphisms. For such maps, computing the $1$\textsuperscript{st} dynamical degree is a significantly more tractable problem, their theory of canonical heights is much more developed, and powerful results exist to study their dynamics and verify conjectures \cite{MR4070310, MR4574221, MR3742591,2408.00566}. 
Unfortunately, the much better theoretical properties of surjective endomorphisms come with a cost: it is much harder to write down explicit examples of such maps, except in a small number of special cases, namely polynomial maps of projective spaces, equivariant morphisms of toric varieties, and isogenies of abelian varieties. For example, it is unknown if a smooth projective Fano variety of Picard number $1$, other than projective space, can have a non-isomorphic surjective endomorphism \cite[Conjecture 1.1]{MR2918167} and \cite{MR1467127,MR1697369,MR1993759, MR4669541,2302.11921}.

Given a projective variety $X$ equipped with a non-isomorphic surjective endomorphism $\varphi\colon X\ra X$, and a morphism $\psi\colon \PP^r\ra \PP^r$, we obtain the product map $\varphi\times\psi\colon X\times \PP^r\ra X\times \PP^r$. Recognizing that $X\times \PP^r$ is the projective bundle $\PP(\O_{X}^{\oplus r+1})\ra X$, the product map fits into a diagram
\[
    \begin{tikzcd}
    \PP(\O_{X}^{\oplus r+1}) \arrow[r, "\psi"] \arrow[d, "\pi" left] & \PP(\O_{X}^{\oplus r+1})\arrow[d, "\pi"] \\ 
    X \arrow[r, "\varphi"] & X.
    \end{tikzcd}
\]
Then, one may hope to obtain non-isomorphic surjective endomorphisms by replacing $\O_{X}^{\oplus r+1}$ with a rank $r+1$ vector bundle $\E$ on $X$ in the above diagram, as there are no immediate conditions which preclude the existence of such a map. In fact, this general setting is of primary interest. Indeed, Satriano and Lesieutre~\cite[Lemma~6.2]{lesieutresatriano2021ksc} have shown that if $\psi\colon \PP(\E)\ra\PP(\E)$ is surjective, then there is a surjective map $\varphi\colon X\ra X$ such that $\pi\circ \psi^m=\varphi\circ\pi$ for $m\gg 1$. 

Additionally, the classification of smooth projective surfaces admitting a non-isomorphic surjection given in \cite{MR2154100} suggests that projective bundles play a large role when classifying more general projective varieties admitting a non-isomorphic surjective endomorphism. Moreover, the work of Amerik~\cite{Amerik1} and Amerik--Kuznetsova~\cite{Amerik2} suggests that this is a tractable problem. Finally, there are approaches to holomorphic dynamics involving the minimal model program, where projective bundles arise naturally as Mori-fibre spaces. From this perspective, understanding surjective endomorphisms of projective bundles is a crucial step, since they represent a terminal step of the minimal model program.

Our first result concerns such endomorphisms when $\E$ is a direct sum of line bundles. 

\begin{theorem}\label{thm:mainLineBundle}
    Let $X$ be projective variety and $\varphi\colon X\ra X$ be a surjective endomorphism. Consider line bundles $\L_i$ on $X$ for each $0\leq i\leq r$, and define $\E=\bigoplus_{i=0}^r\L_i$. Let $\pi\colon \PP(\E)\ra X$ denote the natural projection. For each partition $\lambda:=(\lambda_0,\ldots, \lambda_r)$ of $d$, define $\L^{\lambda}:=\bigotimes_{i=0}^r\L_{i}^{\otimes\lambda_i}$. Then, the two following pieces of data are equivalent:
    \begin{enumerate}
    \item A surjective endomorphism $\psi\colon \PP(\E)\ra \PP(\E)$ such that $\varphi\circ\pi=\pi\circ \psi$ and $\psi^*\O_{\PP(\E)}(1)\cong \O_{\PP(\E)}(d)$ for some positive integer $d$.
    \item A collection of $r+1$ polynomials
    \[
        F_i(x,\mathbf{z})=\sum_{\vert \lambda\vert=d}s_{i,\lambda}(x)\mathbf{z}^{\lambda},
    \]   
    such that for each point $p\in X$, the evaluations $F_i(p,\mathbf{z})=\sum_{\vert \lambda\vert=d}s_{i,\lambda}(p)\mathbf{z}^{\lambda}$ define an endomorphism of $\PP^{r}_{\mathbf{z}}$. Here $s_{i,\lambda}\in H^0(X,\L^{\lambda}\otimes \varphi^*\L_i^{-1})$ and $\mathbf{z}^{\lambda}$ is the product $\prod_{i=0}^rz_i^{\lambda_i}$ of the variables $z_0,\ldots,z_r$.
    \end{enumerate}
    The relation between both of the previous perspectives is that $\psi$ restricted to a fibre $\pi^{-1}(p)$ is given by $[F_0(p,\mathbf{z}),\ldots, F_r(p,\mathbf{z})]$.
\end{theorem} 

The classification in Theorem~\ref{thm:mainLineBundle} is a consequence of the fact that we can easily compute the transition functions of a direct sum of line bundles. The utility of this approach was first used by the second and third-named authors in~\cite{NasserdenZotine2023}. As an application, we explicitly classify all surjective morphisms of Hirzebruch surfaces.

\begin{corollary}
Consider the Hirzebruch surface $\mathbb{F}_n=\PP(\O_{\PP^1}\oplus \O_{\PP^1}(n))$ where $n\geq 1$.

\begin{enumerate}
\item Let $\psi\colon \mathbb{F}_n\ra \mathbb{F}_n$ be a surjective endomorphism. Then, there is some surjective morphism $\varphi\colon \PP^1\ra \PP^1$ such that $\pi\circ \psi=\varphi\circ \pi$.

\item Let $\varphi\colon \PP^1\ra \PP^1$ be any surjective morphism on $\PP^1$. Then, morphisms $\psi$ making the following diagram commute 
\[
    \xymatrix{\mathbb{F}_n\ar[d]_\pi\ar[r]^\psi & \mathbb{F}_n\ar[d]^\pi\\ \PP^1\ar[r]_{\varphi} & \PP^1}
\]
are given by pairs
\[
    F_1([x:y],z_1,z_2)=\sum_{i=0}^ds_i([x:y])z_1^{d-i}z_2^i,\quad F_2([x:y],z_1,z_2)=cz_2^d,
\]
where $c, s_0$ are non-zero constants, and $s_i\in \O_{\PP^1}(ni)$ for all $i=1,\dots, d$. Concretely, for a point in $(p,[z_1:z_2])\in\mathbb{F}_n$, the map $\psi$ is given by
\[
    (p,[z_1:z_2])\mapsto \left(\varphi(p), \left[\sum_{i=0}^ds_i(p)z_1^{d-i}z_2^i:cz_2^d\right]\right).
\] 
\end{enumerate}
\end{corollary}

Since Hirzebruch surfaces are toric surfaces, it is well known that they admit non-isomorphic equivariant surjections. The previous corollary describes all non-equivariant morphisms explicitly as well. Much like projective space, we see that the majority of the morphisms are non-equivariant, that is, non-monomial. Moreover, the explicit representation of these maps shows that every such map preserves the divisor $\{z_2=0\}$, and when thinking of these maps as a family of morphisms parametrized by $\mathbb{P}^1$, every morphism in the family is totally ramified at $\infty=[1:0]$.

Moving beyond line bundles, Theorem~\ref{thm:mainLineBundle} and the Hirzebruch surface examples show that to study surjective endomorphisms of $\mathbb{P}(\mathcal{E})$ we must understand the finer geometry of the bundle $\mathbb{P}(\mathcal{E})$ explicitly. Two natural families of bundles where this is possible are vector bundles on curves and equivariant vector bundles on toric varieties. The case of vector bundles on elliptic curves was addressed in~\cite{NasserdenZotine2023}; here we study the latter case.

Let $X$ be a smooth projective toric variety and $\mathcal{E}$ an equivariant vector bundle on $X$. It is well-known \cite[Proposition A.2]{MR3491049} that $\mathbb{P}(\mathcal{E})$ is a toric variety if and only if $\mathcal{E}$ is a direct sum of line bundles. Therefore, for general equivariant bundles, $\mathbb{P}(\mathcal{E})$ is not toric. Nevertheless, we may exploit the torus-equivariant structure of $\mathcal{E}$ to deduce properties of any surjective map $\psi\colon \mathbb{P}(\mathcal{E})\ra \mathbb{P}(\mathcal{E})$. 

Consider the commutative diagram
\begin{equation}
\label{eq:dagger0}
    \begin{tikzcd}
    \mathbb{P}(\mathcal{E}) \arrow[r, "\psi"] \arrow[d, "\pi" left] & \mathbb{P}(\mathcal{E})\arrow[d, "\pi"] \\ 
    X \arrow[r, "\varphi"] & X,
    \end{tikzcd}
\tag{$\dagger$}
\end{equation}
where $\psi$ and $\varphi$ are surjective. While the dynamical degrees $\lambda_1(\psi)$ and $\lambda_1(\varphi)$ govern the dynamical complexity of the maps, it is important to consider the dynamical complexity of the restriction of $\psi$ to the fibres. This leads to the notion of the \emph{relative dynamical degree} $\lambda_1(\psi\vert_\pi)$. The relation between these three dynamical degrees is given by $\lambda_1(\psi)=\max(\lambda_1(\varphi),\lambda_1(\psi\vert_\pi))$. Let $F$ denote a general fibre of $\pi$. Then, the relative dynamical degree is simply the degree of the polynomials defining the map of projective spaces given by $\psi\vert_F\colon F\ra \varphi(F)$, which does not depend on the chosen fibre. For convenience we call this the \emph{relative degree} of $\psi$. See Subsection~\ref{subsec:on-dynamics} for more details. 

Our second result relates these degrees for equivariant bundles over toric varieties.

\begin{theorem}\label{thm:mainSoftToric}
    Let $X$ be a smooth projective toric variety, and let $\mathcal{E}$ be an equivariant bundle on $X$. Assume that $\mathcal{E}$ is not isomorphic to $\mathcal{L}^{\oplus r}$ for any line bundle $\mathcal{L}$ on $X$. If $\psi\colon \mathbb{P}(\mathcal{E})\ra \mathbb{P}(\mathcal{E})$ is surjective and $\varphi\colon X\ra X$ is an equivariant surjective morphism making diagram \eqref{eq:dagger0} commute, then $\lambda_1(\varphi)=\lambda_1(\psi)$. In other words, the morphisms on the fibres are given by homogeneous polynomials of degree $\lambda_1(\varphi)$.     
\end{theorem}

The restriction on $\mathcal{E}$ excludes the trivial case $\mathbb{P}(\mathcal{E})\cong X\times \mathbb{P}^{\textnormal{rank}(\mathcal{E})-1}$, while the condition that $\pi\circ \psi=\varphi\circ \pi$ with $\varphi$ equivariant provides a weak form of equivariance for $\psi$.

A typical application is when $\varphi\colon X\ra X$ is the toric morphism obtained by extending the map $(t_1,\ldots,t_n)\mapsto (t_1^d,\ldots,t_n^d)$ on the torus of $X$ to all of $X$. In this case, $\lambda_1(\varphi)=d$, and the map $\psi$ restricted to the fibres (projective spaces) is given by homogeneous polynomials of degree $d$.

As a corollary of Theorem~\ref{thm:mainSoftToric}, we obtain the Kawaguchi--Silverman conjecture for such projective bundles. Assume that $X$ and $\psi$ are defined over $\overline{\mathbb{Q}}$. Recall that any ample line bundle $\mathcal{L}$ on a projective variety defines a logarithmic height function $h_\mathcal{L}$, which measures the arithmetic complexity of a point. For each point $p\in X(\overline{\mathbb{Q}})$, Kawaguchi and Silverman defined the arithmetic degree as $\alpha_\psi(p):=\lim_{n\ra \infty}h_H(\psi^n(p))^{1/n}$ in \cite{kawaguchisilverman2016degrees,MR3233709} and showed that this limit exists and is independent of the choice of ample line bundle. When $p$ has a Zariski dense orbit under $\psi$, they conjectured \cite[Conjecture 6]{kawaguchisilverman2016degrees} that $\alpha_\psi(p)=\lambda_1(\psi)$. In other words, the arithmetic complexity equals the geometric complexity. See Subsection~\ref{subsec:on-dynamics} for further discussion.

\begin{corollary}\label{cor:mainKSC}
    Let $\mathcal{E}$ be an equivariant vector bundle over a smooth projective toric variety $X$. Consider a surjective endomorphism $\psi:\mathbb{P}(\mathcal{E})\to\mathbb{P}(\mathcal{E})$ such that there exists a toric morphism $\varphi:X\to X$ making diagram \eqref{eq:dagger0} commute. Then, the Kawaguchi--Silverman conjecture holds for $\psi$.
\end{corollary}
While Theorem~\ref{thm:mainSoftToric} allows us to conclude interesting properties of surjective endomorphisms of projective toric bundles, it does not directly give any information about the existence or non-existence of such maps. In particular, for a given equivariant vector bundle $\mathcal{E}$, the theorem may be vacuously true, as $\mathbb{P}(\mathcal{E})$ could have no surjective endomorphisms at all. The results of Amerik \cite{Amerik1}, Amerik--Kuznetsova \cite{Amerik2}, and Nasserden--Zotine \cite{NasserdenZotine2023} suggest that if $\lambda_1(\psi\vert_\pi)>1$, then $\mathcal{E}$ splits as a direct sum of line bundles. Combining this heuristic with Theorem~\ref{thm:mainSoftToric} suggests that it is difficult for $\mathbb{P}(\mathcal{E})$ to have any surjective endomorphisms at all.

We verify this heuristic by classifying surjective endomorphisms of the tangent and cotangent bundles of a smooth projective toric variety that satisfy the weak equivariance condition of Theorem~\ref{thm:mainSoftToric}.

\begin{theorem}\label{thm:mainTC}
    Let $X$ be a smooth projective toric variety that is not $(\mathbb{P}^1)^n$ for some $n\geq 1$, and let $\mathcal{E}$ be either the tangent bundle or cotangent bundle of $X$. If $\psi\colon \mathbb{P}(\mathcal{E})\ra \mathbb{P}(\mathcal{E})$ is a surjective morphism such that $\pi \circ \psi^m=\varphi\circ \pi$ for some equivariant surjective endomorphism $\varphi\colon X\ra X$, then $\psi$ is an automorphism. 
\end{theorem}

We must exclude products of $\mathbb{P}^1$ in the above result as these toric varieties are precisely those for which the tangent bundle and cotangent bundle are direct sums of line bundles. We treat the case of $(\mathbb{P}^1)^n$ separately in Example~\ref{ex:PP1-classification}. In order to prove Theorem~\ref{thm:mainTC}, we require access to the transition functions of the tangent and cotangent bundles, which can be computed directly. With the transition functions in hand, one may reduce the existence of a surjective morphism to a question of combinatorial commutative algebra which can be directly attacked. To extend these results to other equivariant vector bundles requires explicitly computing the transition functions of these bundles.\\

The remainder of this paper is organized as follows. Section~2 introduces the relevant background and provides further references for the interested reader. Theorem~\ref{thm:mainSoftToric} is proved in Section~3, together with other structural results pertaining to surjective endomorphisms of projective equivariant bundles over toric varieties. In Section~4 we introduce the \emph{transition function method}, which is the main technical tool allowing us to classify surjective endomorphisms of projective bundles, and is used extensively in the subsequent sections. Section~5 proves Theorem~\ref{thm:mainLineBundle} and applies it to classify surjective morphisms of Hirzebruch surfaces. Finally, Section~6 establishes Theorem~\ref{thm:mainTC}, proving the non-existence of non-trivial surjective endomorphisms for the projectivizations of the tangent and cotangent bundles of smooth toric varieties. Throughout we work over $\overline{\bQ}$.

\section{Background}
\label{sec:background}
\subsection{Dynamics}
\label{subsec:on-dynamics}
%%%%% Important:
%%%%% The following was copied from Brett-Sasha's other paper. We should edit it to adapt it for our purposes
%%%%%
We discuss some tools for studying surjective maps of projective bundles. Fix $X$ to be a normal projective variety over an algebraically closed field and let $\E$ be a vector bundle on $X$. Central to the study of surjective morphisms of projective bundles are commutative diagrams
\begin{equation}
    \label{eq:dagger1}
    \begin{tikzcd}[ampersand replacement = \&]
    \bP(\E) \arrow[r, "\psi"] \arrow[d, "\pi" left] \& \bP(\E) \arrow[d, "\pi" right] \\
    X \arrow[r, "\varphi" below] \& X
    \end{tikzcd}
    \tag{$\dagger$}
\end{equation}
where $\psi$ and $\varphi$ are surjective morphisms. Following \cite[Tag 01OB]{stacks-project}, we use $\O_{\bP(\E)}(1)$ to denote the canonical quotient of $\pi^* \E$; the restriction of $\O_{\bP(\E)}(1)$ to a fibre of $\pi$ is $\O_{\bP^{\rank \E}}(1)$. The centrality of such diagrams is due to the following.

\begin{proposition}{\cite[Lemma~6.2]{lesieutresatriano2021ksc}}\label{prop:reductiontodagger}
If $X$ is a smooth projective variety and $\psi\colon \bP(\E)\to \bP(\E)$ is any surjective endomorphism, then there is some integer $n\geq 1$ and a surjective endomorphism $\varphi\colon X\to X$ such that the diagram
\begin{align*}
\xymatrix{\bP(\E)\ar[r]^{\psi^n}\ar[d]_\pi & \bP(\E)\ar[d]^\pi\\ X\ar[r]_\varphi & X} 
\end{align*}
commutes.
\end{proposition}
In fact this holds when $\pi$ is a so called Mori-fibre space, which appear as terminating steps in the minimal program. This proposition allows one to replace the morphism $\psi\colon \bP(\E)\to \bP(\E)$ with some iterate so that we have a commutative diagram \eqref{eq:dagger1}. As many dynamical properties are preserved by iteration, we may often assume that there exists a diagram \eqref{eq:dagger1}. Such an assumption allows us to break the  dynamics $\psi$ into two pieces: the dynamics of $\varphi$ and the dynamics of the induced morphisms on the fibres of $\pi$. To measure the dynamical complexity of $\psi$ on the fibres, we use the relative dynamical degree.

\begin{definition}[{\cite[Definition 2.1]{lesieutresatriano2021ksc}}]
    \label{def:relDyndeg}
    Let $X$ and $Y$ be normal projective varieties. Suppose that we have a diagram
    \[\xymatrix{
    X\ar[r]^\psi\ar[d]_\pi & X\ar[d]^\pi \\ Y\ar[r]_\varphi & Y}\]
    where $\psi,\varphi$, and $\pi$ are surjective morphisms. Fix ample divisors $H$ on $X$ and $W$ on $Y$. We define the \textit{dynamical degree of $\psi$ relative to $\pi$} by the formula
    \[
        \lambda_1(\psi\vert_\pi) \coloneqq \lim_{n\to \infty}\left((\psi^n)^*H\cdot (\pi^*W^{\dim Y})\cdot H^{\dim X-\dim Y-1}\right)^{\frac{1}{n}}.
    \]
\end{definition}

\begin{definition}\label{def:dynDeg}
Take $Y$ to be the spectrum of the base field and $\pi$ the structure map of $X$. In this case we obtain the usual notion of  dynamical degree
\begin{equation*}
\lambda_1(\psi)=\lim_{n\ra \infty}\left((\psi^n)^*H\cdot H^{\dim X-1}\right)^{\frac{1}{n}}
\end{equation*}

\end{definition}

These limit exists and is independent of $H$ and $W$ by \cite[Theorem~1.1]{truong2020relative}. The first dynamical degree is closely connected to the relative dynamical degree. In particular, \cite[Theorem~2.2.2]{lesieutresatriano2021ksc} gives $\lambda_1(\psi)=\max\{\lambda_1(\varphi),\lambda_1(\psi\vert_\pi)\}$.

When we have a diagram \eqref{eq:dagger1}, the fibres of $\pi$ are projective spaces, and the relative dynamical degree can be related to the degree of $\psi$ restricted to the fibres of $\pi$. 

\begin{proposition}[{\cite[Proposition 2.3]{NasserdenZotine2023}}]\label{prop:fibreProp}
    Let $X$ be a normal projective variety and let $\E$ be a vector bundle on $X$. Suppose that we have a diagram of the form \eqref{eq:dagger1}. Then,
    \[
        \psi^*(\O_{\bP(\E)}(1))\cong\O_{\bP(\E)}(\lambda_1(\psi\vert_\pi))\otimes\pi^*\B
    \]
    for some line bundle $\B$ on $X$. 
\end{proposition}

Observe that Proposition~\ref{prop:fibreProp} further shows that the relative dynamical degree is an integer. This is further emphasized by the following proposition, which shows that the relative dynamical degree is equal to the degree of $f$ on the fibres of $\pi$.

\begin{proposition}[{\cite[Proposition 2.4]{NasserdenZotine2023}}]\label{prop:fibreProp2}
    Let $X$ be a normal projective variety and let $\E$ be a vector bundle on $X$. Suppose that we have a diagram of the form \eqref{eq:dagger1}. Then, the relative dynamical degree $\lambda_1(\psi\vert_\pi)$ is the dynamical degree of $\psi$ on the fibres of $\pi$.
\end{proposition}

The possible degree on the fibres of $f$ depends on the geometry of $\E$. For example if $\E$ is trivial then $\bP(\E)=\bP^{\rank \E-1}\times X$ and any degree on the fibres is possible. 

\begin{proposition}\label{sec bg: degrees for endos of pn}
    Let $\varphi:\bP^n\to\bP^n$ be a surjective map. Then, $\varphi$ is finite of degree $d^n$, where $\varphi^*\O_{\bP^n}(1)=\O_{\bP^n}(d)$. Moreover, in this case the first dynamical degree of $\varphi$ is $\lambda_1(\varphi)=d$.
\end{proposition}
\begin{proof}
This follows from the definition of the first dynamical degree and the fact that a morphism to a projective space is given by $n+1$ sections of $\O_{\bP^n}(q)$ for some $q$. This completely determines the degree.
\end{proof}

 Now let $X$ be a normal projective variety defined over $\overline{\mathbb{Q}}$ and $H$ an ample divisor on $X$. Associated to $H$ is a Weil height function $h_H$ that measures the arithmetic complexity of the point $P\in X(\Qbar)$ relative to $H$. If $\varphi\colon X\ra X$ is a surjective endomorphism Kawaguchi and Silverman defined the \textit{arithmetic degree} of $P$ with respect to $\varphi$ as
\begin{equation*}
	\alpha_\varphi(P):=\lim_{n\rightarrow \infty}h^+_{H} \bigl( \varphi^n(P) \bigr)^{\frac{1}{n}},
\end{equation*}
where $h^+_{H} \bigl( \varphi^n(P) \bigr)=\max\{1, h_{H} \bigl( \varphi^n(P) \bigr)\}$. See \cite{kawaguchisilverman2016degrees} for the details of this construction. The arithmetic degree measures the growth of heights along a forward orbit of $\varphi$. Kawaguchi and Silverman conjecture a strong relationship between these two distinct notions of complexity. 

\begin{conjecture}[{\cite[Conjecture 6]{kawaguchisilverman2016degrees}}]\label{conj:KSC}
	Let $X$ be a normal projective variety defined over $\Qbar$ and let $\varphi\colon X\ra X$ be a surjective endomorphism. Let $P\in X(\Qbar)$ be a point such that the forward orbit $O_f(P)$ is Zariski dense in $X$. Then $\alpha_\varphi(P)=\lambda_1(\varphi)$.
\end{conjecture}

We attack Conjecture~\ref{conj:KSC} using the following well known result, which allows one to reduce the Kawaguchi-Silverman conjecture for a fibration, to the Kawaguchi-Silverman conjecture on the base of the fibration in certain conditions. 

\begin{proposition}[{\cite[Corollary 3.2]{lesieutresatriano2021ksc}\label{cor:standard1}}]
Suppose that we have a commuting diagram of normal projective varieties defined over $\Qbar$
    \[\xymatrix{X\ar[r]^\psi\ar[d]_\pi & X\ar[d]^\pi\\ Y\ar[r]_\varphi & Y}\]
     where $\psi,\varphi,\pi$ are all surjective morphisms. When the Kawaguchi--Silverman conjecture holds for $\varphi$ and $\lambda_1(\psi)=\lambda_1(\varphi)$, the Kawaguchi--Silverman conjecture holds for $\psi$. 
\end{proposition}

\subsection{Toric geometry}

%List of toric concepts that are used.\todo remove
%\begin{itemize}
    %\item N and M lattices (done)
    %\item Orbit-cone correspondence (done)
    %\item Affine toric opens for each cone (done)
    %\item Every divisor is linearly equivalent to an invariant one (done)
    %\item Polytope of a line bundle/divisor. Global sections of this line bundle. (done)
    %\item Support function of a line bundle/divisor (done)
    %\item Classification of surjective endos of toric map (done)
    %\item Notation for multiplication by d map $[d]$ (done)
    %\item Definition of a toric vector bundle (done)
   % \item Brief mention of classification results (done)
  %  \item Structural results for line bundles and toric vector bundles over affine toric varieties (done)
 %   \item Decomposable v.b. and toric varieties (done)
%\end{itemize}

We will assume some familiarity with the basics of toric geometry. Here we briefly summarize the relevant definitions and facts, and refer the interested reader to \cite{FultonToric} or \cite{CLSToric} for detailed accounts

A toric variety is an irreducible normal algebraic variety $X$ containing an algebraic torus $T$ as a Zariski open subset, such that the action of $T$ on itself extends to an algebraic action of $T$ on $X$. We denote by $N=\operatorname{Hom}(\mathbb{G}_m,T)$ the lattice of one-parameter subgroups of $T$, and by $M=\operatorname{Hom}(T,\mathbb{G}_m)$ the character lattice of $T$. These lattices are dual via a natural pairing $\langle \cdot, \cdot \rangle: M \times N \to \mathbb{Z}$.

A strongly convex rational polyhedral cone $\sigma$ in $N_\mathbb{R} = N \otimes_\mathbb{Z} \mathbb{R}$ is the set of all nonnegative $\mathbb{R}$-linear combinations of finitely many integral vectors $v_1, \ldots, v_n \in N$, with the property that $\sigma$ contains no nonzero linear subspace of $N_\mathbb{R}$. We write $\sigma = \langle v_1, \ldots, v_n \rangle$ for the cone generated by these vectors. A fan $\Sigma$ in $N_\mathbb{R}$ is a finite collection of strongly convex rational polyhedral cones that is closed under taking faces and such that any two cones intersect along a common face. We denote by $\Sigma(k)$ the set of all $k$-dimensional cones in $\Sigma$.

There is a bijective correspondence between toric varieties and fans in $N_\mathbb{R}$. A fan $\Sigma$ is complete if the union of its cones equals $N_\mathbb{R}$, and smooth if every cone in $\Sigma$ is generated by part of a $\mathbb{Z}$-basis of $N$. A toric variety is complete (resp. smooth) if and only if its associated fan is complete (resp. smooth).

A toric variety is a disjoint union of its orbits by the action of its torus. There is a unique such orbit $O_\sigma$ for every cone $\sigma\in\Sigma$. This correspondence is inclusion-reversing, so that $O_\sigma\subseteq O_\tau$ if and only if $\tau$ is a face of $\sigma$ and, moreover, $\dim \sigma = \operatorname{codim} O_\sigma$. The same holds true upon taking the closure of the orbits of the action. In particular, the rays of $\Sigma$ correspond to $T$-invariant divisors, while maximal cones correspond to $T$-fixed points. Since every divisor on a toric variety is linearly equivalent to a torus-invariant divisor, both the Picard and class groups are generated by the $T$-invariant divisors.

The orbit decomposition induces an affine open cover of $X$. Each cone $\sigma\in\Sigma$ determines a $T$-invariant affine open subvariety $U_\sigma \subseteq X$ given by $U_\sigma = \bigcup_{\tau \subseteq \sigma} O_\tau$. Its coordinate ring is the semigroup algebra $k[\sigma^\vee \cap M]:=k[\,\chi^\mathbf{u}\,\vert\, \mathbf{u}\in\sigma^\vee\cap M]$, where $\sigma^\vee$ denotes the dual cone of $\sigma$.

For computational purposes, it is useful to describe divisors and their sections explicitly. The data of a $T$-invariant Cartier divisor on $X$ is equivalent to that of a continuous piecewise linear function $\varphi\colon \vert \Sigma\vert\ra \bR$ from the support of $\Sigma$. A $T$-invariant Cartier divisor $D$ also determines a rational convex polyhedron in $M_\mathbb{R}$, which is given by
\[
    P_D:=\{m\in M_\bR\colon \varphi(u)\leq \langle m,u\rangle\textnormal{ for all }u\in \vert\Sigma\vert\}.
\]
This polytope encodes the global sections of $\O(D)$:
\[
    H^0(X_\Sigma,\O(D))=\bigoplus_{m\in P_D\cap M}\bC\cdot \chi^m.
\]
A morphism of toric varieties $\varphi: X \to Y$ is a morphism such that $\varphi(T_X)\subseteq T_Y$ and $\varphi|_{T_X}$ is a group homomorphism. Toric maps correspond bijectively to lattice homomorphisms $\phi: N_X \to N_Y$ such that for every cone $\sigma \in \Sigma_X$, the image $\phi(\sigma)$ is contained in some cone of $\Sigma_Y$.

Of particular importance for us will be the following result. While this lemma is likely familiar to experts, we have been unable to locate a suitable reference and therefore include a proof.

\begin{lemma}\label{lemma:toric reduction to mult by d}
Let $X$ be a smooth complete toric variety with fan $\Sigma$ in the lattice $N$, and let $\varphi: X \to X$ be a surjective toric endomorphism. Denote by $\overline{\varphi}: N \to N$ the induced $\mathbb{Z}$-linear map on the lattice. Then, there exists an integer $m \geq 1$ such that either
\begin{enumerate}
\item $\overline{\varphi^m}=\overline{\varphi}^m = d \cdot \text{Id}_{N}$ for some positive integer $d$, or
\item $X$ decomposes as $X = X_1 \times \cdots \times X_t$, where each $X_i$ is an indecomposable toric variety with fan $\Sigma_i$ in the lattice $N_i$, and
\[
\overline{\varphi}^m = d_1 \cdot \text{Id}_{N_1} \times d_2 \cdot \text{Id}_{N_2} \times \cdots \times d_t \cdot \text{Id}_{N_t}
\]
for some positive integers $d_1, \ldots, d_t$, where $\Sigma = \Sigma_1 \times \cdots \times \Sigma_t$ and $N = N_1 \times \cdots \times N_t$.
\end{enumerate}
\end{lemma}

\begin{proof}
    The map $\varphi:X\to X$ is surjective if and only if its induced map $\overline{\varphi}:N\to N$ is surjective \cite[Proposition~2.1.4 and Lemma 2.1.6]{hu2000toric}. Since $\overline{\varphi}$ is a $\mathbb{Z}$-linear map, then $\dim\overline{\varphi}(\sigma)\leq\dim\sigma$ for every cone $\sigma\in\Sigma$. Therefore, since $\overline{\varphi}$ is surjective, it defines a dimension-preserving bijection among the cones in $\Sigma$. 

    By the previous paragraph, there exists a sufficiently divisible integer $m$ such that $\overline{\varphi}^m(v)=d v$ for all primitive generators of the one-dimensional cones in $\Sigma$, and some positive scalar $d$ depending on $v$. If this scalar is the same for every ray, then $\overline{\varphi}^m$ is a multiple of the identity and claim (1) in the statement follows.

    Let us then suppose that $\overline{\varphi}^m$ has at least two distinct eigenvalues, say $d_1,\dots,d_t$. Let $N_1,\dots,N_t\subseteq N$ be the eigenspaces corresponding to the eigenvalues $d_1,\dots,d_t$ of $\overline{\varphi}^m$, so that $N=N_1\times\cdots\times N_t$. For each eigenvalue $d_i$ define the following subfan of $\Sigma$:
    \[
        \Sigma_{i} = \{\zeta\in\Sigma\,\vert\,\zeta\subseteq (N_{i})_\mathbb{R}\}.
    \]
    We claim that $\Sigma = \Sigma_{1}\times\cdots\times\Sigma_{t}$, that is, given cones $\sigma_i\in\Sigma_i$ for all $i=1,\dots,t$, the Minkowski sum $\sigma=\sigma_1+\cdots+\sigma_t$ is an element of $\Sigma$, and every cone of $\Sigma$ arises this way.

    Consider a cone $\sigma\in\Sigma$ and let $S=\{\mathbf{v}_1,\dots,\mathbf{v}_n\}$ be the set of all primitive generators of its rays. Then, the sets $S_i=\{\mathbf{v}\in S\,\vert\, \mathbf{v}\in N_i\}$ define a nontrivial partition of $S$. Since $\sigma$ is simplicial, the cone generated by any subset of $S$ defines a face of $\sigma$. In particular, the cone $\sigma_i:=\langle \mathbf{v}\,\vert\, \mathbf{v}\in S_i\rangle\in\Sigma_i$ is a face of $\sigma$ for all $i=1,\dots,t$. It follows that $\sigma = \sigma_1+\cdots+\sigma_t$.

    Conversely, consider cones $\sigma_i\in\Sigma_i$ for $i=1,\dots,t$, not all of which are the zero cone. We must show that the cone $\sigma_1+\cdots+\sigma_t$ is in $\Sigma$. Consider a nonzero vector $\mathbf{v}\in N_\mathbb{R}$ in the relative interior of $\sigma_1+\cdots+\sigma_t$. Since $\Sigma$ is complete, every vector in $N_\mathbb{R}$ lies in the relative interior of a unique cone of $\Sigma$. In particular, there exists a unique cone $\sigma\in\Sigma$ such that $v$ lies in its relative interior. As explained in the previous paragraph, there exist unique cones $\sigma_i'\in\Sigma_i$ for all $i=1,\dots,t$ such that $\sigma = \sigma_1'+\cdots+\sigma_t'$. Then, there exist some strictly positive constants $c_i,c_i'\in\mathbb{R}$ such that
    \[
        \mathbf{v} = c_1\mathbf{v}_1+\cdots+c_t\mathbf{v}_t = c_1'\mathbf{v}_1'+\cdots+c_t'\mathbf{v}_t',
    \]
    where $\mathbf{v}_i\in\operatorname{rel.int.}(\sigma_i)$ and $\mathbf{v}_i'\in\operatorname{rel.int.}(\sigma_i')$.

    For each $i=1,\dots,t$, the vector $c_i\mathbf{v}_i-c_i'\mathbf{v}_i'$ is an eigenvector of $\overline{\varphi}^m$ with eigenvalue $d_i$. Since eigenvectors corresponding to different eigenvalues are linearly independent, we must have $c_i\mathbf{v}_i-c_i'\mathbf{v}_i'=0$ for all $i$. This implies that both $\mathbf{v}_i$ and $\mathbf{v}_i'$ lie in the relative interior of the same cone in $\Sigma$. Given that $\mathbf{v}_i\in\operatorname{rel.int.}(\sigma_i)$ and $\mathbf{v}_i'\in\operatorname{rel.int.}(\sigma_i')$, we conclude that $\sigma_i=\sigma_i'$ for all $i=1,\dots,t$. It follows that $\sigma=\sigma_1+\cdots+\sigma_t\in\Sigma$, completing the proof.   
\end{proof}  

\begin{remark}
    Consider a toric variety $X=X_\Sigma$ with fan $\Sigma$ in $N$, and a surjective toric morphism $\varphi:X\to X$ such that $\overline{\varphi}=d\cdot Id_N$. Let $\sigma\in\Sigma$ be a maximal cone and $U_\sigma$ its corresponding invariant affine open set. Then, $\varphi|_{U_\sigma}$ is also a surjective toric map, and its corresponding ring homomorphism maps $\chi^\mathbf{u}$ to $\chi^{d\mathbf{u}}$ for every $\mathbf{u}\in\sigma^\vee\cap M$.
\end{remark}

\begin{definition}
\label{def: mult by d map}
    Given a toric variety $X$ and a positive integer $d$, we denote by $[d]:X\to X$ the surjective toric endomorphism corresponding to the lattice map $d\cdot Id:N\to N$. We refer to this map as either the multiplication by $d$ map on $X$, or the $d$\textsuperscript{th} Frobenius map on $X$.
\end{definition}

A \emph{toric vector bundle} on a toric variety $X$ with torus $T$ is a vector bundle $\pi:\E\to X$ endowed with a $T$-action that is linear on the fibres and such that $\pi$ is equivariant. In general, neither $\E$ nor $\bP(\E)$ need be toric varieties, even if $\E$ is a toric vector bundle.

Toric vector bundles were first classified by Kaneyama in \cite{kaneyama1975equivariant} using their splitting behavior over affine opens, involving both combinatorial and linear algebraic data. Klyachko later provided an alternative classification via $\mathbb{Z}$-graded filtrations of finite-dimensional $k$-vector spaces satisfying certain compatibility conditions \cite{klyachko1990equivariant}. For further details on the history of the subject see \cite[Section~2.4]{payne2008moduli}.

The tangent and cotangent bundles of a smooth toric variety are both toric vector bundles over it. As stated above, every line bundle is also endowed with the structure of a toric bundle. It follows that a decomposable vector bundle given by the sum of toric line bundles is also a toric vector bundle. In fact, a vector bundle over a toric variety has total space that is a toric variety if and only if it decomposes as a sum of line bundles. Similarly, if $\E$ decomposes as a sum of line bundles, then $\bP(\E)$ is a toric variety. For more details on these results see \cite[Section~7.3]{CLSToric}.

\section{Properties of Based Maps}
\label{sec:based-maps}
In this section we recall the definition for based morphisms of projective bundles and study the dynamical properties of these morphisms as well as some of their consequences. We begin with the definition.

\begin{definition}
\label{defn:based-maps}
Let $X$ be a smooth projective variety and fix a vector bundle $\E$ on $X$ with projection to the base denoted $\pi \colon \E \rightarrow X$. Let $\psi \colon \PP(\E) \rightarrow \PP(\E)$ be a surjective endomorphism. We say $\psi$ is \textit{based} if there exists a surjective morphism $\varphi \colon X \rightarrow X$ such that $\varphi \circ \pi = \psi \circ \pi$, that is, making the following diagram commute: 
\begin{equation}
\label{eq:dagger2}
\begin{tikzcd}[ampersand replacement = \&]
\bP(\E) \arrow[r, "\psi"] \arrow[d, "\pi" left] \& \bP(\E) \arrow[d, "\pi" right] \\
X \arrow[r, "\varphi" below] \& X.
\end{tikzcd}
\tag{$\dagger$}
\end{equation}
We call $\varphi$ the \textit{base map}. There is an integer $d$ and line bundle $\B$ on $X$ such that $\psi^* \O_{\PP(\E)}(1) = \O_{\PP(\E)}(d) \otimes \pi^* \B$. We refer to $d$ as the \textit{degree of $\psi$ on the fibres of $\pi$} or the \textit{relative (dynamical) degree}.

If $X$ is a toric variety, then we say $\psi$ is \textit{base-toric} if the morphism $\varphi$ can be taken to be a toric morphism.
\end{definition}

For every surjective endomorphism $\psi \colon \PP(\E) \rightarrow \PP(\E)$, there is an integer $m \geq 1$ such that the iterate $\psi^m$ is based by \cite[Lemma~6.9]{lesieutresatriano2021ksc}. If we have the diagram \eqref{eq:dagger2}, then for any closed point $x \in X$ we automatically obtain $\psi(\pi^{-1}(x))\subseteq \pi^{-1}(\varphi(x))$. As each fibre is isomorphic to $\PP^n$, this is in fact an equality. In other words, we may think of $\psi$ as a family of maps $\psi\vert_{\pi^{-1}(x)} \colon \PP^r \rightarrow \PP^r$ parametrized by $X$. 

Our central result in this section centres around controlling the dynamics of based maps. Crucial to this control is how our bundle splits on the irreducible (and particularly rational) curves of the base variety.

\begin{definition}
\label{defn:projectively-trivial}
Let $X$ be a smooth projective variety, fix a vector bundle $\E$ on $X$, and suppose $C$ is a closed irreducible curve in $X$. Denote $\E\vert_C$ as the pull-back of $\E$ to the normalization of $C$.
\begin{enumerate}
\item We say that $\E$ is projectively trivial if $\PP(\E) \cong X \times \PP^{\rank(\E)-1}$. Equivalently $\E \cong \L^{\oplus r}$ for some line bundle $\L$ on $X$.
\item We say that $\E$ is projectively trivial over $C$ if $\E\vert_C$ is projectively trivial.
\item When $\varphi \colon X \rightarrow X$ is a morphism, we say $C$ is \textit{periodic for $\varphi$} if there exists an integer $k \geq 1$ such that $g^k(C) = C$.
\end{enumerate}
\end{definition}

We may now state the main result of this section.

\begin{theorem}
\label{thm:basedmaps-plus-projtrivial}
Let $X$ be a normal projective variety and fix a vector bundle $\E$ on $X$. Suppose $\psi \colon \PP(\E) \rightarrow \PP(\E)$ is a based map with base map $\varphi \colon X \rightarrow X$. If there exists an irreducible rational curve $C \subseteq X$ such that $C$ is periodic for $\varphi$ and $\E$ is not projectively trivial over $C$, then we have $\lambda_1(\psi) = \lambda_1(\varphi)$.
\end{theorem}
As a corollary, we obtain the Kawaguchi--Silverman conjecture holds for $\bP(\E)$. Before proving the theorem, we need to establish several technical results---the key to the proof lies in showing that the induced map $\psi^* \colon N^1(\PP(\E\vert_C))\to N^1(\PP(\E\vert_C))$ is a constant linear map. To this end, we first recall some positivity properties of vector bundles.

\begin{definition}
\label{defn:vector-bundle-positivity}
Let $X$ be a normal projective variety and $\E$ a vector bundle on $X$. We say that $\E$ is ample, nef or big whenever the line bundle $\O_{\PP(\E)}(1)$ on $\PP(\E)$ possesses the respective property.
\end{definition}

We need to understand the nef cone of a projective bundle on a rational curve, i.e. on $\PP^1$. The first step is understanding the positivity of $\O_{\PP(\E)}(1)$ over $\PP^1$, something one gleans from the more general result on projective space. See \cite{MR2207204} for some related results from a different perspective.
\begin{lemma}
\label{lem:positivity-of-bundles}
Consider the vector bundle $\E = \bigoplus_{i=1}^r \O_{\PP^n}(d_i)$ over $\PP^n$, where $d_0 \leq d_1 \cdots \leq d_r\in \bZ$ are such that $d_0 = 0$ and some $d_i > 0$. We have that $\E$ is nef and big, but not ample.
\end{lemma}
\begin{proof}
The vector bundle $\E$ is ample (nef) if and only if each of its summands is ample (nef) by \cite[Proposition~6.1.13 and Theorem~6.2.12]{Lazarsfeld-PosII}. Therefore, our assumption on $\E$ tells us it is nef but not ample. It remains to show that that $\E$ is big. This is equivalent to the statement that
\begin{equation*}
    \dim H^0 \bigl( \PP(\E), \O_{\PP(\E)}(d) \bigr) = \dim H^0(\PP^n,\Sym^d \E)\geq C  d^{r-1+n}
\end{equation*}
for some constant $C > 0$ and all integers $d\gg 1$. Let $\widetilde{\E}=\O_{\PP^n}^{\oplus r-1} \oplus \O_{\PP^n}(1)$. It suffices to show that $\widetilde{\E}$ is big since
\begin{equation*}
    \dim H^0(\PP^n, \Sym^d \E) \geq \dim H^0(\PP^n, \Sym^d \widetilde{\E}).
\end{equation*}
By definition, the bigness of $\widetilde{\E}$ over $\PP^n$ is equivalent to that one of $\O_{\PP(\widetilde{\E})}(1)$ over $\PP(\widetilde{\E})$. Define $\zeta \in A \bigl( \PP(\widetilde{\E}) \bigr)$ to be the first Chern class $c_1(\O_{\PP(\widetilde{\E})}(1))$ in the Chow ring. We already know that $\O_{\PP(\widetilde{\E})}(1)$ is nef, so it suffices to show that $\zeta^{\dim \PP(\widetilde{\E})} = \zeta^{r+n-1}$ is non-zero.

Let $\widetilde{\pi} \colon \PP(\widetilde{\E}) \rightarrow \PP^n$ be the projection to the base. Then the pullback $\widetilde{\pi}^* \colon A(\PP^n) \rightarrow A \bigl( \PP(\widetilde{\E}) \bigr)$ defines an injection and from \cite[Theorem~9.6]{eisenbud20163264} it follows that $\widetilde{\pi}$ induces an isomorphism
\begin{equation*}
    A \bigl( \PP(\widetilde{\E}) \bigr) \cong \frac{A(\PP^n)[\zeta]}{\langle \zeta^r + c_1(\widetilde{\E})\zeta^{r-1} \rangle}.
\end{equation*}
Therefore $\zeta^{r} = -c_1(\widetilde{\E})\zeta^{r-1}\in A \bigl( \PP(\widetilde{\E}) \bigr)$, and from this we can deduce that
\begin{equation*}
    \zeta^{r+n-1} = -c_1(\widetilde{\E})\cdot\zeta^{r+n-2} = \cdots = \bigl( -c_1(\widetilde{\E}) \bigr)^n \cdot \zeta^{r-1},
\end{equation*}
Now we know $c_1(\widetilde{\E})=\widetilde{\pi}^*\Bigl(c_1 \bigl( \O_{\PP^n}(1) \bigr) \Bigr)\neq 0$, so this product is non-zero by the presentation of $A \bigl( \PP(\widetilde{\E}) \bigr)$ above. Hence $\O_{\bP(\widetilde{\E})}(1)$ is a nef vector bundle with non-vanishing top intersection power, which implies that $\widetilde{\E}$ is big, and so we conclude that $\E$ is big as well.
\end{proof}

This allows us to understand the pullbacks of morphisms $\psi \colon \PP(\E\vert_C) \rightarrow \PP(\E\vert_C)$.

\begin{lemma}
\label{lem:pullback-is-constant}
Consider the vector bundle $\E = \bigoplus_{i=1}^r\O_{\PP^n}(d_i)$ on $\PP^n$, where the integers $d_0\leq d_1 \leq \cdots\leq d_r$ are such that $d_0 = 0$ and some $d_i > 0$. Suppose $\psi \colon \PP(\E) \rightarrow \PP(\E)$ is a based morphism with base morphism $\varphi$ and projection $\pi \colon \PP(\E) \rightarrow \PP^n$. Then, we have that $\psi^* \colon N^1 \bigl( \PP(\E) \bigr)\to N^1 \bigl( \PP(\E) \bigr)$ acts by scalar multiplication by $d$, where $d$ the unique integer such that $\varphi^* \O_{\PP^n}(1) = \O_{\PP^n}(d)$.
\end{lemma}

\begin{proof}
As $N^1 \bigl( \PP(\E) \bigr)$ is a two-dimensional vector space, it is enough to show that $\psi^*$ has three non-collinear eigenvectors with eigenvalue $d$. In particular, our goal is to show that the rays of the nef cone $\operatorname{Nef} \bigl( \PP(\E) \bigr)$ and the rays of the pseudoeffective cone $\overline{\operatorname{Eff}}\bigl( \PP(\E) \bigr)$ are eigenvectors of $\psi^*$.

Recall that the big cone is the interior of the pseudoeffective cone, and the ample cone is the interior of the nef cone. We claim that $\psi^*$ maps big divisors to big divisors and ample divisors to ample divisors. Indeed, consider any line bundle $\L$ on $\PP(\E)$. The map $\psi$ is surjective, hence necessarily finite by \cite[p.~1]{Amerik1}. It follows that $\psi^*\L$ is big if and only if $\L$ is big by \cite[Lemma~4.3]{holschbach2010chebotarev}, and $\psi^*\L$ is ample if and only if $\L$ is ample by \cite[Proposition~6.1.8.(iii)]{Lazarsfeld-PosII} and \cite[Section~01VG, Lemma~0892]{stacks-project}. In particular, these equivalences imply that $\psi^*$ maps the rays of $\operatorname{Nef} \bigl( \PP(\E) \bigr)$ to rays of $\operatorname{Nef} \bigl( \PP(\E) \bigr)$ and the rays of $\overline{\operatorname{Eff}}\bigl( \PP(\E) \bigr)$ to rays of $\overline{\operatorname{Eff}}\bigl( \PP(\E) \bigr)$.

Now observe that, by Lemma~\ref{lem:positivity-of-bundles}, the bundle $\O_{\PP(\E)}(1)$ is big and nef, but not ample. Thus its class is a boundary ray of $\operatorname{Nef} \bigl( \PP(\E) \bigr)$ but not of $\overline{\operatorname{Eff}}\bigl( \PP(\E) \bigr)$. 

On the other hand, the class of $\pi^*\O_{\PP^n}(1)$ is a ray of both cones. Indeed, it is nef because it is the pullback of an ample divisor, but it is not ample because its intersection with any curve contained in a fibre is zero. Similarly, since $\pi^*\O_{\PP^n}(1)$ is nef, its volume is given by its top self-intersection, which is zero for dimension reasons. This implies that this line bundle is not big. It follows that $\pi^*\O_{\PP^n}(1)$ defines a ray of both the nef and pseudoeffective cones of $\PP(\E)$.

Finally, since $\psi$ is based, we have $\pi \circ \psi = \varphi \circ \pi$. Hence, $\psi^* \pi^* \bigl( \O_{\PP^n}(1) \bigr) = \pi^* \varphi^* \bigl( \O_{\PP^n}(1) \bigr) = \pi^* \bigl( \O_{\PP^n}(d) \bigr)$, so that $\pi^* \O_{\PP^n}(1)$ is an eigenvector with eigenvalue $d$. Since $\psi^*$ is injective, it follows that the other two rays of $\operatorname{Nef} \bigl( \PP(\E) \bigr)$ and $\overline{\operatorname{Eff}}\bigl( \PP(\E) \bigr)$ must also be eigenvectors. But by our observations above, these rays are distinct. We conclude that there are three non-collinear eigenvectors for $\psi^*$ as desired, and since the eigenvalue of one of them is $d$, it follows that $\psi^*$ must be scalar multiplication by $d$.
\end{proof}

Finally, we need the following technical lemma which leverages some arithmetic dynamics.

\begin{lemma}
\label{lem:nonjumping}
Let $\varphi \colon X \rightarrow X$ be a surjective morphism and $V \subseteq X$ a closed subvariety with $\varphi(V)=V$. Assume that the Kawaguchi--Silverman conjecture holds for $\varphi\vert_V$ and that $P \in V$ is a point whose orbit under $\varphi \vert_{V}$ is Zariski dense in $V$. Then, we have that $\lambda_1(\varphi\vert_V)\leq \lambda_1(\varphi)$.
\end{lemma}

\begin{proof}
Let $\iota \colon V \rightarrow X$ be the inclusion map. We have that $\alpha_{\varphi\vert_V}(P)=\lambda_1(\varphi\vert_V)$ as we have assumed the Kawaguchi--Silverman conjecture is true for $\varphi\vert_V$.
Fix an ample divisor $H$ on $X$ and a height function $h_H$ for $H$. It follows that the divisor $\widehat{H} = \iota^*H$ has a height function $h_{\widehat{H}}=h_H \circ \iota$. Therefore, we have that
\begin{equation*}
    \alpha_\varphi(P) = \lim_{n \rightarrow \infty} h_H(\varphi^n(P))^{\frac{1}{n}} = \lim_{n \rightarrow \infty} h_{\widehat{H}} \bigl( (\varphi\vert_V)^n(P) \bigr)^{\frac{1}{n}} = \alpha_{\varphi\vert_V}(P).
\end{equation*}
Moreover, by \cite[Theorem~1.4]{matsuzawa2020bounds}, we always have that $\lambda_1(g)\geq \alpha_g(P)$ so that
\begin{equation*}
    \lambda_1(\varphi) \geq \alpha_\varphi(P) = \alpha_{\varphi\vert_V}(P) = \lambda_1(\varphi\vert_V)
\end{equation*}
as claimed.
\end{proof}

\begin{remark}
The assumption that $V$ is fixed under $\varphi$ is necessary. For example, let $\varphi \colon \PP^2 \rightarrow \PP^2$ be the second toric Frobenius map on the projective plane, so that a point $[x:y:z]$ gets sent to $[x^2:y^2:z^2]$. A straightforward computation shows that the dynamical degree is $\lambda_1(\varphi) = 2$. Consider the rational curve $C = \VV(x^2 + y^2 + z^2)$. The image of $C$ under $\varphi$ is the curve $C' = \VV(x + y + z)$, which is also rational. Hence we may interpret the restriction $\varphi\vert_C$ as a map from $\PP^1$ to $\PP^1$. However, since the dynamical degree over a curve is equal to the topological degree, we have that $\lambda_1(\varphi\vert_C) = 4$, as every point generically has four preimages under $\varphi$. It is a straightforward check that $C$ is not a periodic curve for $\varphi$.
\end{remark}

With the lemmas in hand, we are able to proceed with the proof.

\begin{proof}[Proof of Theorem~\ref{thm:basedmaps-plus-projtrivial}]
Since $\lambda_1(\varphi) = \lambda_1(\varphi^k)$ for any positive integer $k$, we may assume without loss of generality that $\varphi(C) = C$. Fix an isomorphism $C\cong \bP^1$ and write
\begin{equation*}
    \E\vert_C \cong \bigoplus_{i=1}^{\rank \E}\O_C(d_i)
\end{equation*} 
for some integers $d_i$. By our assumption that $\E$ is not projectively trivial over $C$, we must have that $a_i\neq a_j$ for some $i\neq j$. Since $\varphi(C)=C$, and $\psi$ commutes with $\varphi$ after restricting to $C$, we obtain an induced commutative diagram
\begin{equation*}
\begin{tikzcd}
\PP \Bigl(\bigoplus_{i=1}^{\rank \E}\O_C(a_i) \Bigr) \arrow[r, "\psi" above] \arrow[d, "\pi" left] & \PP \Bigl(\bigoplus_{i=1}^{\rank \E}\O_C(a_i) \Bigr) \arrow[d, "\pi" right] \\
\PP^1 \arrow[r, "\varphi" below] & \PP^1.
\end{tikzcd}
\end{equation*}
After possibly twisting $\E\vert_C$ by a line bundle, we may assume that $d_1 \leq d_2 \leq \cdots \leq d_r$ are such that $d_1 = 0$ and some $d_i > 0$.

Suppose that $\psi^*\O_{\PP(\E\vert_C)}(1)=\O_{\PP(\E\vert_C)}(d)$. Lemma~\ref{lem:pullback-is-constant} implies that $\psi^* \colon N^1 \bigl( \PP(\E\vert_C) \bigr) \rightarrow N^1 \bigl( \PP(\E\vert_C) \bigr)$ is the multiplication by $d$ map. This means that $d = \lambda_1(\varphi\vert_C)$ by Proposition~\ref{sec bg: degrees for endos of pn}, and Proposition~\ref{prop:fibreProp} tells us that $\lambda_1(\psi\vert_\pi) = d$. Since $\varphi(C) = C$, Lemma~\ref{lem:nonjumping} tells us that $\lambda_1(g\vert_C)\leq \lambda_1(g)$. Therefore, the product formula for dynamical degrees allows us to conclude that
\begin{equation*}
    \lambda_1(\psi) = \max\{\lambda_1(\psi\vert_{\pi}),\lambda_1(\varphi)\} = \lambda_1(\varphi). \qedhere
\end{equation*}
\end{proof}

One natural setting for Theorem~\ref{thm:basedmaps-plus-projtrivial} is for toric vector bundles on toric varieties. In this case, the torus-invariant curves on the base toric variety are good candidates for checking projective triviality. Biswas and Santos show that for rationally connected varieties, a bundle $\E$ being projectively trivial on all rational curves implies that $\E$ is projectively trivial; see \cite[Theorem~1.1]{biswassantos2009}. Here is a strengthening of this result in the case of toric varieties, where we only demand projective triviality on the torus-invariant curves.

\begin{proposition}
Suppose $\E$ is a toric vector bundle on a smooth projective toric variety $X$. If $\E$ is projectively trivial over every torus-invariant curve on $X$, then $\E$ is projectively trivial.
\end{proposition}

\begin{proof}
Let $C$ be a torus invariant curve. By assumption we have that $\E\vert_C\cong\O_C(a_C)^{\oplus r}$ for some integer $a_C$. We obtain $\det(\E)\vert_C\cong\det(\E\vert_C)\cong\O_C(ra_C)$. Note that $\langle \det(\E),C\rangle=\deg(\det(\E)\vert_C)=ra_C$. Since the torus invariant curves generate $N_1(X)$, we see that $\dfrac{1}{r}\langle \det \E,\bullet\rangle$ gives an integral linear form on $N_1(X)$. Since $X$ is a smooth projective toric variety, \cite[Theorem~10.8]{Danilov1978toric} says the cohomology and Chow rings are isomorphic and torsion-free. Furthermore, recall that \cite[Corollary on p.~65]{Danilov1996cohomology} says that for smooth varieties, the intersection pairing between divisors and curves is unimodular modulo torsion.  Hence, the intersection pairing is an integral perfect pairing. Therefore there is some line bundle $\L$ on $X$ with $\langle \L,C\rangle=a_C$ and $\L^{\otimes r}=\det\E$. Now consider $\E\otimes \L^{-1}$. We have $(\E\otimes \L^{-1})\vert_C\cong \O_C(a_C)^{\oplus r}\otimes \O_C(-a_C)\cong\O_C^{\oplus r}$. In other words $(\E\otimes \L^{-1})\vert_C$ is trivial for any torus invariant curve. By \cite[Theorem~6.4]{hering2010positivity}, this implies $\E\otimes \L^{-1}\cong \O_X^{\oplus r}$, giving that $\E\cong \L^{\oplus r}$ as desired. 
\end{proof}

If the endomorphism $\psi \colon \PP(\E) \rightarrow \PP(\E)$ is base-toric, then the torus-invariant curves on $X$ are automatically periodic, and so Theorem~\ref{thm:basedmaps-plus-projtrivial} applies.

\begin{corollary}\label{cor:maincor}
Suppose $\E$ is a toric vector bundle that is not projectively trivial on a smooth projective toric variety $X$ and $\psi \colon \PP(\E) \rightarrow \PP(\E)$ is a base-toric map with base map $\varphi$. We have $\lambda_1(\psi) = \lambda_1(\varphi)$.
\end{corollary}
Note that \cite[Theorem 4.1]{MR4070310} implies Conjecture~\ref{conj:KSC} for projective toric varieties. Therefore, applying Corollary~\ref{cor:maincor} and Proposition~\ref{cor:standard1} together gives a proof of Corollary~\ref{cor:mainKSC} for the introduction.

\begin{remark}
    Theorem~\ref{thm:basedmaps-plus-projtrivial} may be generalized to higher genus curves. If $C$ is allowed to be an irreducible curve of geometric genus $1$ and $\E\vert_C$ is not a direct sum of torsion line bundles, then we may also conclude that $\lambda_1(\psi)=\lambda_1(\varphi)$. Indeed, by the proof of \cite[Corollary 3.8]{lesieutresatriano2021ksc} we have that either $\lambda_1(\psi\vert_C) = \lambda_1(\varphi\vert_C)$ or $\E\vert_C$ is semistable of degree 0. In the latter case the degree on the fibres is at most $\lambda_1(\varphi\vert_C) \leq \lambda_1(\varphi)$. In the former case $\E\vert_C$ is semistable of degree 0 and is not a direct sum of torsion line bundles by assumption. Hence, we may apply \cite[Theorems~1.1 and 1.2]{NasserdenZotine2023} to obtain that the degree on the fibres is at most $\lambda_1(\varphi)$. In either case we conclude via the product formula that $\lambda_1(\psi) = \lambda_1(\varphi)$.
\end{remark}

% \begin{remark}
% We obtain the Kawaguchi-Silverman surjective \todo was brett writing this?
% \end{remark}

\subsection{Chern classes of bundles admitting based endomorphisms}

Admitting a based endomorphism seems to impose strong conditions on the Chern classes of the vector bundle, which we investigate in this subsection. Let $X$ be a smooth projective variety and $\E$ be a rank $r$ vector bundle on $X$. Any based morphism $\psi \colon \PP(\E) \rightarrow \PP(\E)$ with base map $\varphi$ induces a commutative diagram of Chow rings
\begin{equation*}
\begin{tikzcd}
A^\bullet (X) \arrow[r, "\varphi^*" above] \arrow[d, "\pi^*" left] & A^\bullet (X) \arrow[d, "\pi^*" right] \\
A^\bullet \bigl( \PP(\E) \bigr) \arrow[r, "\psi^*" below] & A^\bullet \bigl( \PP(\E) \bigr).
\end{tikzcd}
\end{equation*}
From now on we embed $A^\bullet(X)$ inside $A^\bullet(\PP(\E))$ via $\pi^*$. In what follows we use additive notation for the Chow ring and identify $\alpha\in \Pic(X)$ with $\pi^*\alpha\in \Pic(\PP(\E))$. Set $\L=\O_{\PP(\E)}(1)$. Then we have $\psi^* \L=d\L+\alpha$ where $\alpha\in \Pic(X)$ and $d$ is the degree of $\psi$ on the fibres of $\pi$. The Chow ring of $\PP(\E)$ is identified by the quotient
\begin{equation*}
    A^\bullet \bigl( \PP(\E) \bigr) = \frac{A^\bullet(X)[\L]}{P_\E(\L)},
\end{equation*}
where $P_\E(\L) = \sum_{i=0}^{r} (-1)^i \L^{r-i} c_i(\E)$. In particular, $\psi^* P_\E(\L) = 0$, so we have that $\psi^* P_\E(\L) = q(\L) P_\E(\L)$ for some $q(\L) \in A^\bullet(X)[\L]$. On the other hand we know that
\begin{equation*}
    \psi^* P_\E(\L) = \psi^* \Bigl( \sum_{i=0}^{r} (-1)^i \L^{r-i} c_i(\E) \Bigr) = \sum_{i=0}^r (-1)^i (d \L + \alpha)^{r-i} \varphi^* c_i(\E).
\end{equation*}
The leading term of the right-hand side is $d^r \L^r$, which is $d^r$ times the leading term of $P_\E(\L)$, so we conclude that $q(\L)$ is a constant polynomial $d^r$. This gives us the equality
\begin{equation}
\label{eqn:chernclass-conditions}
    \sum_{i=0}^r (-1)^i (d \L + \alpha)^{r-i} \varphi^* c_i(\E) = d^r \sum_{i=0}^r (-1)^i \L^{r-i} c_i(\E).
\end{equation}
By comparing coefficients in this equation, we are able to derive conditions on the Chern classes of $\E$.

\begin{lemma}
\label{lem:chern-coefficients}
The coefficient of $\L^n$ in $\psi^* P_\E(\L)$ is
\begin{equation*}
    \sum_{i=0}^{r-n} (-1)^i \binom{r-i}{n} d^n \cdot \alpha^{r-n-i} \cdot \varphi^* c_i(\E).
\end{equation*}
\end{lemma}

\begin{proof}
Sum over powers $(d\L+\alpha)^{r-i}$ with $r-i \geq n$ and apply the binomial theorem.
\end{proof}
This allows us to determine the line bundle $\alpha$ explicitly in terms of $\E$, $d$, and $\varphi$. 

\begin{corollary}
\label{cor:trans factor}
We have that 
\begin{equation*}
    \alpha=\frac{\varphi^*c_1(\E) - dc_1(\E)}{r}.
\end{equation*}
\end{corollary}

\begin{proof}
By Lemma~\ref{lem:chern-coefficients}, the $\L^{r-1}$ term of $\psi^* P_\E(\L)$ has coefficient
\begin{equation*}
    \binom{r}{r-1} \cdot d^{r-1} \cdot \alpha \cdot \varphi^* c_0(\E) - \binom{r-1}{r-1} \cdot d^{r-1} \cdot \varphi^* c_1(\E).
\end{equation*}
The equality in \eqref{eqn:chernclass-conditions} tells us this must be equal to $d^r c_1(\E)$, so we obtain
\begin{equation*}
    r d^{r-1} \alpha - d^{r-1} \varphi^* c_1(\E) = d^r c_1(\E).
\end{equation*}
Solving for $\alpha$ gives the result.
\end{proof}

Finally, when the base map satisfies $\varphi^* c_i(\E) = q^i c_i(\E)$ for some integer $q \neq d$ and all $i \in \mathbb{N}$, we can conclude that all of the Chern classes of $\E$ are generated by the first Chern class of $\E$. Examples of maps satisfying this assumption are the $q$th toric Frobenius map on a toric variety, or multiplication by $q$ on an abelian variety.

\begin{theorem}\label{thm:chernrelations}
Suppose that $\psi$ is a based map with base map $\varphi$ and relative degree $d$. If $\varphi^* c_i(\E) = q^i c_i(\E)$ for some integer $q \neq d$ and all $i \in \mathbb{N}$, then for any $k \geq 1$ we have
\begin{equation*}
    c_k(\E) = \frac{\binom{r}{k} c_1(\E)^k}{r^k}.
\end{equation*}
\end{theorem}

The proof is a direct but tedious combinatorial computation and so we postpone it to the Appendix. Theorem~\ref{thm:chernrelations} shows that under suitable hypotheses, the existence of a surjective morphism with interesting properties (the degree on the fibres is not the dynamical degree on the base) imposes restrictive conditions on the Chern classes of the vector bundle. Taking $\E$ to be the tangent bundle (or cotangent bundle) we obtain equations for the Chern classes of the variety itself. In other words, the base variety must have some special geometry in order to satisfy such equations.

\begin{remark}
\label{rem:chern-vs-dual}
We obtain a relationship between the properties of surjective endomorphisms of $\PP(\E)$ and $\PP(\E^\vee)$. Indeed, assume that $\varphi^* c_i(\E) = q^i c_i(\E)$. Using that $c_i(\E^\vee)=c_i(\E)(-1)^i$ we obtain
\begin{equation*}
    c_k(\E) = \frac{\binom{r}{k} c_1(\E)^k}{r^k}\iff  c_k(\E^\vee) = \frac{\binom{r}{k} c_1(\E^\vee)^k}{r^k}.
\end{equation*}
This is interesting because $\PP(\E)$ and $\PP(\E^\vee)$ may have vastly different geometry. For instance, the tangent bundle on a smooth complete toric variety is always a Mori dream space due to \cite[Theorem~5.8]{HausenSuess2010} but the cotangent bundle need not be; see \cite[Theorem~1.5]{GHPS2012}.
\end{remark}

\begin{remark}
\label{rem:chern-toric-hyp}
The assumptions in Theorem~\ref{thm:chernrelations} are satisfied whenever $\varphi^*\L\cong \L^{\otimes d}$ for any line bundles $\L$ and if all the Chern classes $c_k(\E)$ lie in the image of the multiplication map $A^1(X)^{\oplus k} \rightarrow A^k(X)$ defined by $(\L_1, \L_2, \ldots,\L_k) \mapsto \L_1\cdot \L_2 \cdot \cdots \cdot \L_k$. This is true, for instance, if $X_\Sigma$ is an indecomposable smooth projective toric variety. In particular, we always have $\varphi^*\L\cong \L^{\otimes d}$ after iteration, and $A^k(X_\Sigma)$ is always generated by intersection products of line bundles \cite[Proposition~10.3]{Danilov1978toric}.
\end{remark}

% The converse to Theorem~\ref{thm:chernrelations} does not hold.

% \begin{example}
% Suppose that $\L$ is a line bundle on some variety $X$ and set $\E = \O_X \oplus \L$.  We ask when there is a based map $\psi$ with base map being the identity and relative degree $d = 2$. In this case, the conclusion of Theorem~\ref{thm:chernrelations}
% \begin{align*}
% 4 c_2(\E) = c_1(\L)^2.
% \end{align*}
% Since $\E$ is a direct sum of line bundles, we have $c_2(\E) = 0$ and hence $c_1(\L)^2 = 0$. It follows that $\L$ cannot be ample or big. However, we will see in Example~\ref{ex:torsion-line-bundle} that if $\L$ is two-torsion (hence satisfying $c_1(\L)^2 = 0$), then a based map of relative degree $d = 2$ can indeed exist.
% \end{example}

Theorem~\ref{thm:chernrelations} allows us to control the dynamics of the projectivization of the tangent and cotangent bundle on projective space.

\begin{example}
Suppose $X = \PP^n$, let $\E = \T_{\PP^n}$, and set $\L = \O_{\PP^n}(1)$. Fix a surjective endomorphism $\varphi \colon \PP^n \rightarrow \PP^n$ of degree $q$ and assume $\psi \colon \PP(\E) \rightarrow \PP(\E)$ is a based map with base map $\varphi$ with relative degree $d \neq q$. Observe that the hypothesis of Theorem~\ref{thm:chernrelations} is satisfied because $A^\bullet(\PP^n) = \CC[c_1(\L)]/ \langle c_1(\L)^{n+1} \rangle$ and $\varphi^*\L\cong \L^{\otimes q}$. Therefore Theorem~\ref{thm:chernrelations} gives us
\begin{equation}
\label{eq:chern-example}
    c_k(\E) = \frac{\binom{n}{k} c_1(\E)^k}{n^k}
\end{equation}
for $k \geq 2$. On the other hand, the Euler exact sequence allows us to compute that $c_i(\E) = \binom{n+1}{i} c_1(\L)^i$. Substituting this into Equation~\eqref{eq:chern-example} for $i = k$ and $i = 1$, then canceling out $c_1(\L)^k$ on both sides, we obtain
\begin{equation}
    \binom{n+1}{k} = \frac{\binom{n}{k}(n+1)^k}{n^k}
\end{equation}
for $2 \leq k \leq n$. We conclude that a based map on $\PP(\T_{\PP^n})$ with base $\varphi$ with degree on the fibres $d \neq q$ only exists if this equality holds.

However, by taking $k = n$ we obtain $n+1 = \bigl( 1+\frac{1}{n} \bigr)^n$ which never holds when $n > 1$. Consequently any based map on $\PP(\T_{\PP^n})$ must have degree on the fibres equal to $q$. Equivalently, the dynamical degree of a based map must be equal to the dynamical degree of its base map.
\end{example}

If the base has Picard number one, we may also control the dynamics of projective bundles using Theorem~\ref{thm:chernrelations}.

\begin{example}\label{ex:PicardNumber1LineBundles}
Suppose that $X$ is smooth and projective with Picard number one and fix an ample line bundle $\L$. Let $\E = \bigoplus_{i=1}^r \L_i$ where $\L_i = \L^{\otimes n_i} \otimes \N_i$ for some algebraically trivial line bundles $\N_i$ and integers $n_i$, with $r \leq \dim X$. Assume we have a based map of $\PP(\E)$ base map being the identity and having relative degree $d > 1$. If at least two of the integers $n_i$ are distinct, then after twisting by a large enough power of $\L$ we may without loss of generality assume that $n_i \geq 0$ for all $i$, $n_1 = 0$, and $n_j > 0$ for some $j$. In this case we have $c_r(\E) = 0$ as $\E$ has a non-vanishing section, while $c_1(\E) \neq 0$ by construction. Therefore Theorem~\ref{thm:chernrelations} allows us to conclude that no such map exists.

On the other hand, if all the $n_i$ are the same then by twisting we may assume that each $\L$ is algebraically trivial. In this case, \cite[Theorem~1.2]{NasserdenZotine2023} implies the relative degree (which is the same as the dynamical degree) is $d = 1$ unless each $\L_i$ is torsion. If each $\L_i$ is torsion, then in fact we can find some based maps; see Example~\ref{ex:torsion-line-bundle} or \cite[Example~6.1]{NasserdenZotine2023}.
\end{example}

\section{The transition function method}
\label{sec:transition-function-method}
Let $\mathbb{K}$ be an algebraically closed field of characteristic zero. Every surjective morphism of $\PP^n$ corresponds to a collection of $n+1$ polynomials in $\mathbb{K}[x_0,x_1,\ldots,x_n]$ which do not have a common zero. In this section we present an analogous surjectivity criterion for the relative setting of a projective bundle $\PP(\E)$ in terms of the transition functions of $\E$. 

The key tool for translating between endomorphisms of projective bundles and polynomial data is the pushforward-pullback adjunction. Given a morphism $\psi: \PP(\E) \to \PP(\E)$ over a base morphism $\varphi: X \to X$, the existence of the commutative square~\eqref{eq:dagger2} is equivalent to having a surjective morphism of sheaves $\Theta: \pi^*\varphi^*\E \to \psi^*\O_{\PP(\E)}(1)$. In general, when $\psi$ has degree $d$ on the fibres, we have $\psi^*\O_{\PP(\E)}(1) \cong \O_{\PP(\E)}(d) \otimes \pi^*\B$ for some line bundle $\B$ on $X$, so $\Theta$ becomes a surjection $\Theta: \pi^*\varphi^*\E \to \O_{\PP(\E)}(d) \otimes \pi^*\B$. The adjunction 
\begin{equation*}
\operatorname{Hom}_{\PP(\E)} \bigl( \pi^*\varphi^*\E,\O_{\PP(\E)}(d) \otimes \pi^*\B \bigr) \cong \operatorname{Hom}_{X} \bigl(\varphi^*\E,\pi_*(\O_{\PP(\E)}(d) \otimes \pi^*\B) \bigr)
\end{equation*}
combined with the projection formula $\pi_* \bigl( \O(d) \otimes \pi^*\B \bigr) = \Sym^d\E \otimes \B$, translates this into a morphism $\theta: \varphi^*\E \to \Sym^d\E \otimes \B$ on the base.

We may explicitly observe the adjunction formula. Given a trivializing open $U \subseteq X$ such that $\E|_U \cong \O_U^{r}$ and $\B|_U \cong \O_U$, we get $\O_{\pi^{-1}(U)} \cong \O_{U \times \PP^r}$. Hence, when we restrict $\Theta$ to $\pi^{-1}(U)$ we get a morphism $\O_{\pi^{-1}(U)}^{r} \rightarrow \O_{\pi^{-1}(U)}(d)$ mapping $\mathbf{e}_i \mapsto f_i = \sum_{|\lambda|=d} a_{i,\lambda} \mathbf{z}^\lambda$ with $a_{i,\lambda} \in \O_U$. Each monomial $\mathbf{z}^\lambda = z_1^{\lambda_1} z_2^{\lambda_2} \cdots z_r^{\lambda_r}$ with $|\lambda| = \lambda_1 + \cdots + \lambda_r = d$ corresponds to the basis element $\mathbf{e}_1^{\lambda_1} \cdots \mathbf{e}_r^{\lambda_r}$ of $\Sym^d \O_U^{r}$. Therefore the polynomial $f_i \in \O_U[z_1, \ldots, z_r]_d$ is canonically identified with the element $\sum_{|\lambda|=d} a_{i,\lambda} \mathbf{e}_1^{\lambda_1} \cdots \mathbf{e}_r^{\lambda_r}$ in $\Sym^d \O_U^{r}$. In other words, the adjunction identifies $\Theta\vert_{\pi^{-1}(U)}$ with $\theta|_U: \O_U^{r} \to \Sym^d \O_U^{r}$ defined by $\mathbf{e}_i \mapsto \sum_{|\lambda|=d} a_{i,\lambda} \mathbf{e}_1^{\lambda_1} \cdots \mathbf{e}_r^{\lambda_r}$.

Consider a point $(p, [v]) \in \PP(\E)$ where $p \in U$ and $[v] = [v_1 : \cdots : v_r] \in \PP^{r-1}$ is a point in the fibre over $p$. By Nakayama's lemma, the morphism $\Theta$ is surjective if and only if the induced map on fibres $\Theta_p \otimes \kappa(p)$ is surjective at every point $p$. Since the target fibre is one-dimensional over the residue field $\kappa(p, [v])$, this map is surjective if and only if at least one of $f_1(p)(v), \ldots, f_r(p)(v)$ is nonzero. Since this must hold for every point $[v] \in \PP^{r-1}$ in the fibre over $p$ and for every $p \in U$, the polynomials $f_1(p), \ldots, f_r(p)$ have no common zero in $\PP^{r-1}$ for any $p \in U$.

Conversely, suppose we have a map $\theta: \varphi^* \E \to \Sym^d \E \otimes \B$ locally defined by polynomials $f_1, \ldots, f_r \in \O_U[z_1, \ldots, z_r]_d$ such that for every $p \in U$, the polynomials $f_1(p), \ldots, f_r(p)$ have no common zero in $\PP^{r-1}$. The adjunction produces a morphism $\Theta: \pi^* \varphi^* \E \to \O_{\PP(\E)}(d) \otimes \pi^* \B$ defined by these same polynomials. At each point $(p, [v]) \in \PP(\E)$, the non-vanishing condition ensures that at least one of $f_1(p)(v), \ldots, f_r(p)(v)$ is nonzero in $\kappa(p, [v])$. Since the target fibre is one-dimensional, the induced map on fibres is surjective at $(p, [v])$. Once again by Nakayama's lemma, since the induced map on fibres is surjective at every point, $\Theta$ is surjective as a morphism of sheaves.

To define a global morphism, the polynomial data on different trivializing opens must satisfy compatibility conditions dictated by the transition functions. This is the content of the following proposition, which was originally proven in  \cite[Proposition~2.6]{NasserdenZotine2023}.

\begin{proposition}
\label{prop:transition function method}
Let $X$ be a normal projective variety over $\mathbb{K}$ and $\E$ a vector bundle on $X$ of rank $r$. Let $\varphi: X \to X$ and $\psi: \PP(\E) \to \PP(\E)$ be surjective morphisms fitting into a commutative square
\begin{equation}
\label{eq:dagger3}
\begin{tikzcd}
\PP(\E) \arrow[r, "\psi"] \arrow[d, "\pi" left] & \PP(\E) \arrow[d, "\pi" right] \\
X \arrow[r, "\varphi"] & X
\end{tikzcd}
\tag{$\dagger$}
\end{equation}
such that $\psi^*\O_{\PP(\E)}(1) \cong \O_{\PP(\E)}(d) \otimes \pi^*\B$ for some $d \geq 1$ and line bundle $\B$. Let $\{U_j\}$ be an open cover of $X$ trivializing $\E$, $\varphi^* \E$, and $\B$. Then, the morphism $\psi$ exists if and only if there exist, for each $j$, an $r$-tuple of degree $d$ polynomials
\[
    \theta_j = (f_{j,1}, f_{j,2}, \ldots, f_{j,r}) \in \O_X(U_j)[z_1, z_2, \ldots, z_r]^{r}
\]
such that:
\begin{enumerate}
    \item For all $p \in U_j$, the polynomials $f_{j,i}(p)$ have no common zero in $\PP^{r-1}$, and
    \item For all $i, j$, the compatibility diagram
\begin{equation}
\label{eq:compatibility}
\begin{tikzcd}
\O_{U_i\cap U_j}^{r} \arrow[r, "\theta_i" above] \arrow[d, "\varphi^* M_{j,i}" left] & \Sym^d \O_{U_i \cap U_j}^{r} \arrow[d, "\beta_{i,j} \Sym^d M_{j,i}" right] \\  
\O_{U_i\cap U_j}^{r} \arrow[r, "\theta_j"] & \Sym^d \O_{U_i \cap U_j}^{r}
\end{tikzcd}
\end{equation}
commutes, where $M_{j,i}$ is the transition function from $U_i$ to $U_j$ for $\E$, $\varphi^* M_{j,i}$ the transition function for $\varphi^* \E$, and $\beta_{j,i}$ the transition function for $\B$.
\end{enumerate}
\end{proposition}

\begin{remark}
\label{rem:compatibility}
The operator $\Sym^d M$ acts on a polynomial $f \in \O_U[z_1, z_2, \ldots, z_r]$ by $(\Sym^d M)(f) = f \circ M^{\textsf{T}}$. This is because the identification $\Sym^d \O_U^{r} \cong \O_U[z_1, z_2, \ldots, z_r]$ realizes polynomials as functions on the dual bundle $\E^\vee|_{U}$. This means the transition function of $\Sym^d \E$ from $U$ to $V$ is the pullback of the transition function of $\E^\vee$ from $V$ to $U$. The order of the open sets reverses because dualizing reverses the direction of arrows. The formula then follows because the transition function for $\E^\vee$ from $V$ to $U$ is $M^{\textsf{T}}$.
\end{remark}

\begin{proof}[Proof of Proposition~\ref{prop:transition function method}]
By \cite[Proposition~7.12]{hartshorne}, the existence of the commutative square \eqref{eq:dagger3} is equivalent to the existence of a surjective morphism of sheaves $\Theta: \pi^*\varphi^*\E \twoheadrightarrow \psi^*\O_{\PP(\E)}(1)\cong \O_{\PP(\E)}(d) \otimes \pi^*\B$. As explained in the discussion above, by the pullback-pushforward adjunction, $\Theta$ corresponds to a unique morphism
\[
    \theta: \varphi^*\E \to \Sym^d\E \otimes \B.
\]
Since we chose $U_j$ to trivialize all three bundles, we may consider the restrictions
\[
    \theta_j:= \theta|_{U_j}: \O_{U_j}^{r} \to \Sym^d \O_{U_j}^{r}.
\]
We can identify $\Sym^d \O_{U_j}^{r}$ with $\O_X(U_j)[z_1, z_2, \ldots, z_r]_d$, where the variables $z_i$ are dual to the standard basis of $\E|_{U_j}$ and are independent of the choice of open set. The map $\theta_j$ is determined by where it sends the standard basis vectors: $\theta_j(\mathbf{e}_k) = f_{j,k}$ for degree $d$ homogeneous polynomials $f_{j,k} \in \O_X(U_j)[z_1, z_2, \ldots, z_r]$. By the discussion preceding this proposition, the surjectivity of $\Theta$ over $\pi^{-1}(U_j)$ is equivalent to the polynomials $f_{j,1}, \ldots, f_{j,r}$ having no common zero in the fibre over any $p \in U_j$.

Finally, because the $\theta_j$ arise as restrictions of a global morphism $\theta$, they must satisfy compatibility conditions on overlaps. On $U_i \cap U_j$, the transition functions relate the trivializations, and the fact that $\theta$ is a global section of $\mathcal{H}om(\varphi^*\E, \Sym^d\E \otimes \B)$ ensures that diagram \eqref{eq:compatibility} commutes.

Conversely, given a collection of polynomial tuples $\theta_j:\O_{U_j}^{r} \to \Sym^d \O_{U_j}^{r}$ satisfying the non-vanishing and compatibility conditions, they glue to define a global morphism $\theta: \varphi^*\E \to \Sym^d\E \otimes \B$. By the adjunction, this corresponds to a morphism $\Theta: \pi^*\varphi^*\E \to \O_{\PP(\E)}(d) \otimes \pi^*\B$, and the non-vanishing condition ensures that $\Theta$ is surjective, yielding the desired morphism $\psi$.
\end{proof}

\section{The split case}
\label{sec:split-case}
In this section, we aim to classify based maps for projectivizations of split bundles. We begin with the general setting, then specialize to the base-toric case and give explicit examples. Let $X$ be a smooth projective variety, fix line bundles $\L_1, \L_2 \ldots, \L_r$ on $X$, and set $\E=\bigoplus_{k=1}^r\L_k$. Choose an open cover $\{U_{i}\}$ of $X$ trivializing $\E$ (and hence each $\L_k$). Following the conventions in the literature, such as \cite[Section~1.2]{LePotier1997} or \cite[Chapter 6]{CLSToric}, we use $M_{j\leftarrow i}^{(k)} = M_{ji}^{(k)}$ to denote the transition function of $\L_k$ from $U_i$ to $U_j$, so that the cocycle condition is $M_{\ell j}^{(k)} M_{ji}^{(k)} = M_{\ell i}^{(k)}$. Let $\varphi \colon X \rightarrow X$ be a surjective morphism such that $\varphi^*\L_k\cong \L_k^{\otimes q}$ for each $k$ and some nonnegative integer $q$. We aim to classify based maps $\psi \colon \PP(\E) \rightarrow \PP(\E)$ with base map $\varphi$ and relative degree $d$. This means we have a commutative diagram
\begin{equation}
\begin{tikzcd}[ampersand replacement = \&]
    \PP(\E) \arrow[r, "\psi"] \arrow[d, "\pi" left] \& \PP(\E) \arrow[d, "\pi" right] \\
    X \arrow[r, "\varphi" below] \& X.
\end{tikzcd}
\tag{$\dagger$}
\end{equation}
Equation~\ref{eq:compatibility} in Proposition~\ref{prop:transition function method} specifies that, for each trivializing opens $U_i$ and $U_j$ with $j\neq i$, we need $r$ degree $d$ homogeneous polynomials $f_{i,\ell} \in \O_X(U_i)[z_1,\ldots,z_r]$ for $1 \leq \ell \leq r$ satisfying that
\[
    \Sym^d M_{j,i}(\theta_i) = \left(\Sym^d M_{j,i}(f_{i,1}),\dots,\Sym^d M_{j,i}(f_{i,r})\right) = \theta_j \left(\varphi^* M_{j,i}\right),
\]
since we are supposing that $\mathcal{B}=\mathcal{O}_X$ and hence $\beta_{j,i}$ is the identity map. Then, following Remark~\ref{rem:compatibility}, we know that 
\[
    \Sym^d M_{j,i}(f_{i,\ell}) = f_{i,\ell} \left( z_1 M_{ji}^{(1)}, z_2 M_{ji}^{(2)}, \ldots, z_r M_{ji}^{(r)} \right).
\]
Putting these together we obtain that, for all $i\neq j$, 
\begin{equation*}
    f_{i,\ell} \left( z_1 M_{ji}^{(1)}, z_2 M_{ji}^{(2)}, \ldots, z_r M_{ji}^{(r)} \right) = \left( M_{ji}^{(\ell)} \right)^q f_{j,\ell},
\end{equation*}
since the transition matrix of $\E$ from $U_i$ to $U_j$ is the diagonal matrix $\operatorname{diag} \bigl( M_{ji}^{(1)}, M_{ji}^{(2)}, \ldots,M_{ji}^{(r)} \bigr)$, and $\varphi^*\mathcal{L}_{\ell} = \mathcal{L}^{\otimes q}_\ell$. Note that $\theta_j$ acts on $\varphi^* M_{j,i}$ via left matrix multiplication as a row vector.

Writing 
\[
    f_{i,\ell} = \sum_{\vert \lambda\vert=d}a_{i\ell,\lambda} \mathbf{z}^\lambda,
\]
where $a_{i\ell,\lambda} \in \O_X(U_i)$, we have explicitly that
\[
    f_{i,\ell} \bigl( z_1 M_{ji}^{(1)}, z_2 M_{ji}^{(2)}, \ldots, z_r M_{ji}^{(r)} \bigr) = \sum_{\vert \lambda\vert=d} a_{i\ell,\lambda} \bigl( \textbf{M}_{ji} \bigr)^\lambda \mathbf{z}^\lambda,
\]
where $\bigl( \textbf{M}_{ji} \bigr)^\lambda$ denotes the product $\prod_{k=1}^r \bigl( M_{ji}^{(k)} \bigr)^{\lambda_i}$. Now fix $1\leq \ell \leq r$ and some exponent vector $\lambda$. Then, we obtain that 
\[
    \bigl( \textbf{M}_{ji} \bigr)^\lambda a_{i\ell,\lambda} = \bigl( M_{ji}^{(\ell)} \bigr)^q a_{j \ell,\lambda}.
\]
In other words, the collection of sections $\{a_{i\ell,\lambda}\}_i$ is the gluing data of a global section of the line bundle $\L^\lambda\otimes \L_\ell^{\otimes -q}$, where $\L^\lambda \coloneqq \bigotimes_{k=1}^r\L^{\lambda_i}$. Let us summarize this discussion as a theorem:
\begin{theorem}
\label{thm:linebundlecase}
    Let $X$ be a smooth projective variety, suppose $\E = \bigoplus_{i=1}^r \L_i$ is a direct sum of line bundles on $X$, and fix an open cover $\{U_i\}$ of $X$ trivializing $\E$. Assume we have a base map $\varphi \colon X \rightarrow X$ satisfying $\varphi^* \L \cong \L^{\otimes q}$ for some integer $q \geq 0$. Then, describing a based map $\psi$ of $\PP(\E)$ with relative degree $d$ is equivalent to specifying a collection of global sections $a_{\ell,\lambda} \in \Gamma \bigl( X, \L^\lambda\otimes \L_\ell^{\otimes -q} \bigr)$ for every $1 \leq \ell \leq r$ and any $\lambda$ a partition of $d$ such that, for each $i=1,\dots,r$, the polynomials
    \[
    f_{i,\ell} = \sum_{\vert \lambda\vert=d} (a_{\ell,\lambda}\vert_{U_i})\mathbf{z}^\lambda,\qquad 1\leq \ell \leq r,
    \]
    do not have a common zero on $U_i \times \PP^{r-1}$.
\end{theorem}
% For each $\ell=1,\dots,r$ consider the translated section ring 
% \begin{equation*}
%     R_{\ell}(X,\L_1,\ldots,\L_r):=\sum_{\mu\in \mathbb{N}^r}H^0(X,\L^\mu\otimes \L_\ell^{\otimes-d})\mathbf{z}^\mu.
% \end{equation*}
% Put another way, giving a surjective endomorphism of $\PP(\E)$ is equivalent to giving a collection of homogeneous degree $d$ (in the $z$ variables) elements $f_\ell \in R_{\ell}(X,\L_1,\dots,\L_r)$ for $1\leq \ell \leq r$, having no common zero on any fibre of $\pi:\PP(\E)\to X$. 
This idea has already had fruitful applications in the study of surjective endomorphisms of projective bundles on elliptic curves; see \cite[Section~5]{NasserdenZotine2023}.

\begin{example}
\label{ex:frobenius-on-Fn}
Take $X = \PP^1$ and $\E = \O_{\PP^1} \oplus \O_{\PP^1}(n)$ for some positive integer $n > 0$, so that the projectivization $\PP(\E) = \mathbb{F}_n$ is the $n$\textsuperscript{th} Hirzebruch surface. The projection map $\pi:\mathbb{F}_n\to\PP^1$ is a toric map between toric varieties. Moreover, for every positive integer $d$, the map $\pi$ commutes with the $d$\textsuperscript{th} toric Frobenius map $[d]$ on the base and the bundle (see Definition~\ref{def: mult by d map}).
%As this is a toric variety, we have that $\mathbb{F}_n$ possess the $d$\textsuperscript{th} toric Frobenius mapping which we denote by $[d]$. Observe that $[d]$ is a based map with base map $\varphi$ being the $d$\textsuperscript{th} toric Frobenius on $\PP^1$. 
Let us study this example in the framework of Theorem~\ref{thm:linebundlecase}.
   
In this case, since $[d]^* \L = \L^{\otimes d}$ we have that $q = d$. For every pair of nonnegative integers $\lambda = (\lambda_1, \lambda_2)$ with $\lambda_1+\lambda_2=d$, we need to select two global sections
\[
    a_{1,\lambda} \in \Gamma \left(X, \O_{\PP^1}^{\otimes \lambda_1} \otimes \O_{\PP^1}(n)^{\otimes \lambda_2} \otimes \O_{\PP^1}^{\otimes -d} \right) 
    \qquad \text{and} \qquad 
    a_{2,\lambda} \in \Gamma \left(X, \O_{\PP^1}^{\otimes \lambda_1} \otimes \O_{\PP^1}(n)^{\otimes \lambda_2} \otimes \O_{\PP^1}(n)^{\otimes -d} \right).
\]
This is the same as choosing global sections
\begin{align*}
    a_{1,\lambda} \in \Gamma \bigl(X, \O_{\PP^1}(\lambda_2 n) \bigr) && \text{and} && a_{2,\lambda} \in \Gamma \bigl(X, \O_{\PP^1}(\lambda_2 n - dn) \bigr).
\end{align*}
There are two natural choices of global sections: when $\lambda_2 = 0$, the left bundle is trivial and so we can choose $a_{1,\lambda} = 1$. When $\lambda_2 = d$, the right bundle is trivial so we can choose $a_{2,\lambda} = 1$. Setting every other $a_{j,\lambda} = 0$ for $j = 1, 2$ and every $\lambda$, we get that $f_{i,1} = z_1^d$ and $f_{i,2} = z_2^d$ for each $i$ indexing an open set in our cover. These two polynomials do not have a common zero on the $\PP^1$-fibres, since $f_{i,1} = f_{i,2} = 0$ implies $z_1 = z_2 = 0$ which is not a point in $\PP^1$. It follows that this is a valid choice according to Theorem~\ref{thm:linebundlecase}, hence defining a based map $\psi$ with base map $\varphi$. But we also observe that this choice of $f_{i,1}$ and $f_{i,2}$ describes the $d$\textsuperscript{th} toric Frobenius map on the fibres.
\end{example}

If $X$ admits nontrivial torsion line bundles, then we are always able to construct non-trivial based maps.

\begin{example}
\label{ex:torsion-line-bundle}
Suppose $\L$ is a 2-torsion line bundle on $X$, set $\E = \O_{X} \oplus \L$, and let $d = 2$, matching the torsion degree. Let $\varphi$ be the identity on $X$, so that $q = 1$. Similar to the previous example, we need to specify global sections
\begin{align*}
    a_{1,(2,0)} \in \Gamma \bigl(X, \O_X \bigr) && a_{1,(1,1)} \in \Gamma \bigl(X, \L \bigr) && a_{1,(0,2)} \in \Gamma \bigl(X, \L^{\otimes 2} \bigr) \\
    a_{2,(2,0)} \in \Gamma \bigl(X, \L^{\otimes -1} \bigr) && a_{2,(1,1)} \in \Gamma \bigl(X, \O_X \bigr) && a_{2,(0,2)} \in \Gamma \bigl(X, \L \bigr).
\end{align*}
For example it is possible to choose the constant sections $a_{1,(2,0)} = a_{1,(0,2)} = a_{2,(1,1)} = 1$ and all others zero. This gives the polynomials $f_{i,1} = z_1^2 + z_2^2$, $f_{i,2} = z_1 z_2$ for each $i$ indexing an open set in our cover. As these have no common zero on the fibres, we obtain a based morphism of $\PP(\E)$ with base map being the identity and given by $[f_1:f_2]$ on the fibres.

% On the other hand, assume $\L$ is non-trivial. Assuming we are given a based map  $\psi$ over the identity with $\psi^*\O_{\PP(\E)}(1)=\O_{\PP(\E)}(2)$. We obtain that on the fibres it is of the form

% \[[\alpha_1z_1^2+\alpha_2z_1z_2+\alpha_3z_2^2:\beta_1z_1^2+\beta_2z_1z_2+\beta_3z_2^2]\] where $\alpha_i\in H^0(X,\L^{\otimes i})$ for $1\leq i\leq 3$ and $\beta_1\in H^0(X,\L^{-1}),\beta_2\in H^0(X,\O_X)$ and $\beta_3\in H^0(X,\L)$. By example~\ref{ex:torsionlinebundle1} $\L$ is numerically trivial and so has a non-zero global section if and only it is trivial. Therefore if $\L$ is not 2-torsion we have that $\alpha_2,\alpha_3,\beta_1,\beta_3=0$. Therefore our map is of the form \[[\alpha_1z_1^2:\beta_2z_1z_2]\]
% which has a common zero, a contradiction. We obtain that $\L$ must be 2-torsion.

% If we take $\L$ to be torsion of degree $d>1$ then taking $\E=\O_X^{\oplus d-1}\oplus \L$ we can find a degree $d$ morphism on the fibres by a similar argument, specifically we take $f_i=z_i^{d}+z_d^{d}$ for $1\leq i\leq d-1$ and $f_d=z_1\cdot\ldots\cdot z_d$.

% This construction behaves exactly like a map on $X\times \PP^r$ due to the fact that the sections arise from torsion line bundles, and hence are all constant, so there can be no dependence on the base. 

One observes that this map behaves exactly like a map on $X \times \PP^1$ since the sections are arising from powers of the torsion line bundle. In particular, this has no dependence on the base variety $X$.
\end{example}

\subsection{Specialization to smooth projective toric varieties}

Let $X_\Sigma$ an $n$-dimensional smooth projective toric variety associated to a fan $\Sigma$, with $\Sigma(1)$ denoting the collection of rays of $\Sigma$. Fix equivariant line bundles $\L_1,\ldots \L_r$ on $X_\Sigma$ and set $\E=\bigoplus_{k=1}^r\L_k$.  As exemplified in Example~\ref{ex:frobenius-on-Fn}, the bundle $\PP(\E)$ always admits surjective endomorphisms due to its toric structure. We seek to describe a wide class of surjective endomorphisms using Theorem~\ref{thm:linebundlecase} in more detail. 

Recall that there is a one-to-one correspondence between equivariant line bundles $\L$ onf $X_\Sigma$ and piecewise-linear continuous maps $a_\L\colon \vert \Sigma\vert \rightarrow \bR$. This defines a polytope
\begin{equation*}
    P_\L:=\{m\in M_\bR\colon a_\L(m)\leq \langle m,u\rangle\textnormal{ for all }u\in \vert\Sigma\vert\},
\end{equation*}
and moreover we obtain
\begin{equation*}
    H^0(X_\Sigma,\L)=\bigoplus_{m\in P_\L\cap M}\bC\cdot \chi^m.
\end{equation*}

Now fix $d\geq 1$ and let $\varphi \colon X_\Sigma\ra X_\Sigma$ be multiplication by $d$ map. By Theorem~\ref{thm:linebundlecase}, constructing a surjective endomorphism $\psi:\PP(\E)\to\PP(\E)$ amounts to choosing sections $a_{k,\lambda}\in H^0(X_\Sigma,\L^\lambda\otimes\L_k^{-d})$ such that the polynomials 
\[
    f_{\ell} = \sum_{\vert \lambda\vert=d} a_{\ell,\lambda} \mathbf{z}^\lambda, \qquad 1 \leq \ell \leq r,
\]
do not have a common zero. In the toric setting, each $a_{\ell,\lambda}$ is a linear combination of monomials $\chi^{\bf{v}}$ as $\bf{v}$ ranges along the lattice points of $P_{\L^\lambda \otimes \L_\ell^d} \cap M$ for every $1 \leq \ell \leq r$. Finding the lattice points of these polytopes is a potentially computationally intensive, yet combinatorially interesting task. The configuration of the line bundles can significantly limit the options for a non-vanishing collection. Taking $f_{k} = z_k^d$ is always one option; this corresponds to the multiplication by $d$ map on $\PP(\E)$, as we have seen in Example~\ref{ex:frobenius-on-Fn}. On the other hand, one may ask if this is the only option. Indeed, this happens for the tangent bundle of $(\PP^1)^n$ as we see in the following example, but does not happen for other split bundles; see Example~\ref{ex:PP2-partial-classification}.

\begin{question}
    Let $X_\Sigma$ be a smooth projective toric variety and $\E$ an equivariant split vector bundle over $X_\Sigma$. When does $\PP(\E)$ admit a based map with base map $[d]:X_\Sigma\to X_\Sigma$, other than the $d$\textsuperscript{th} toric Frobenius map itself?
\end{question}

\begin{example}
\label{ex:PP1-classification}
Set $X=\left(\PP^1\right)^n$. We use Theorem~\ref{thm:linebundlecase} to analyze the base-toric maps of $\PP(\T_X)$, that is, based maps where the base map $\varphi$ is a toric morphism. In that case, a toric morphism of $X$ is of the form $\varphi = \sigma \circ (\varphi_1 \times \ldots \times \varphi_n)$, where $\varphi_i\colon \PP^1\ra \PP^1$ is of degree $d_i\geq 1$ and $\sigma$ is a permutation in $\mathfrak{S}_n$. Without loss of generality we may assume the permutation is trivial---it is just permuting the factors and does not affect our computation. Since $\T_X=\pi_1^*O_{\PP^1}(2)\oplus\cdots\oplus \pi_n^*O_{\PP^1}(2)$, constructing a based map of $\PP(\T_X)$ with relative degree $d$ is equivalent to specifying a collection of global sections 
\[
    a_{\ell,\lambda} \in H^0 \bigl(X, \O_X(2\lambda_1\mathbf{e}_1+\cdots+2\lambda_n\mathbf{e}_n)\otimes \varphi^*\O_X(-2\mathbf{e}_\ell)\bigr)
\]
for each collection $\lambda=(\lambda_1,\dots,\lambda_n)$ of positive integers with $|\lambda|=\lambda_1+\cdots+\lambda_n=d$, and such that, for each $i=1,\dots,n$, the polynomials
\[
        f_{\ell} = \sum_{\vert \lambda\vert=d} a_{\ell,\lambda}\mathbf{z}^\lambda,\qquad 1\leq \ell \leq n,
\]
do not have a common zero on any fibre. Since pullbacks commute with direct sums, we have that $\varphi^*\O_X(-2\mathbf{e}_\ell) =\O_X(-2d_\ell\mathbf{e}_\ell)$. Thus, we are looking at global sections of 
\[
    \bigoplus_{\ell=1}^n\bigoplus_{|\lambda|=d}\O_X\left(\sum_{i=1}^n2\lambda_i\mathbf{e}_i-2d_\ell\mathbf{e}_\ell\right).
\]
The sections of the $\ell$\textsuperscript{th} summand corresponding to a given $\lambda$ are the coefficients of the the monomial $\mathbf{z}^\lambda$ in the polynomial $f_\ell$ defined above. For any fixed $\ell$ and $\lambda$, K\"unneth's formula gives
\[
    H^0\left(X,\O_X\left(\sum_{i=1}^n 2\lambda_i\mathbf{e}_i-d_\ell\mathbf{e}_\ell\right)\right) = H^0 \bigl(\PP^1,\O_{\PP^1}(2\lambda_\ell-2d_\ell) \bigr) \otimes \left(\bigotimes_{i\neq \ell}^n H^0(\PP^1, \O_{\PP^1}(2\lambda_i))\right)
\]
By Theorem~\ref{thm:linebundlecase}, the only nonzero monomials of $f_\ell$ are those for which $H^0\bigl(X,\O_{\PP^1}(2\lambda_\ell-2d_\ell)\bigr)\neq 0$, which only happens whenever $\lambda_\ell \geq d_\ell$. Since $|\lambda| = d$, we must conclude that $d \geq d_\ell$ or else $f_\ell = 0$ and the collection $f_1, \ldots, f_n$ would automatically have a common zero by B\'ezout's Theorem. Furthermore, since $d_\ell \geq 1$ we get that $f_\ell$ is divisible by $z_\ell$ for all $\ell$. In particular, on any given fibre, we have that $f_1,\dots,f_{\ell-1},f_{\ell+1},\dots,f_n$ vanish at the point $z_1 = z_2 = \cdots = z_{l-1} = z_{l+1} = \cdots = z_n = 0$ and $z_\ell = 1$. Moreover, the coefficient of the monomial $z_\ell^{d_\ell}$ of $f_\ell$ must be nonzero, as otherwise $f_\ell$ would additionally vanish at this point and hence all of $f_1,\dots,f_n$ would vanish at a common point, contradicting the existence of $\psi$. 

Finally, suppose that we had a strict inequality $d > d_\ell$ for some $\ell$, say without loss of generality $\ell = 1$. By our discussion above, we know that each of $f_2, f_3, \ldots, f_n$ vanish at the point $[1 \colon 0 \colon \cdots \colon 0]$, and $f_1$ does not. In particular, we have that $f_1\bigl( [1 \colon 0 \colon \cdots \colon 0] \bigr) = A_1$ for some global section $A_1 \in \O_{\PP^1}(2d - 2d_\ell)$. However, since $2d - 2d_\ell > 0$ by assumption, there is a point $Q \in \PP^1$ where $A$ vanishes. In other words, in the fibre above $[1 \colon 0 \colon \cdots \colon 0]$, all of the polynomials $f_1, f_2, \ldots, f_n$ vanish at the point $Q$, contradicting the existence of $\psi$. Consequently we conclude that $d = d_\ell$ for all $\ell$ and therefore $f_\ell = A_\ell z_\ell^d$ for some constant $A_\ell$.

In conclusion, the only base-toric maps $\psi \colon \PP(\T_X) \rightarrow \PP(\T_X)$ with base map $\varphi = \varphi_1 \times \cdots \times \varphi_n$ are those for which the degree of $\varphi_i$ are all equal to the same integer $d$ and $\psi$ restricted to any fibre is the $d$\textsuperscript{th} toric Frobenius map. One should observe that virtually the same argument works for the cotangent bundle, but the inequalities involved will flip.
\end{example}

\begin{example}
\label{ex:PP2-partial-classification}
Let $X = \PP^2$ and consider the bundle $\E=\O_{\PP^2}(\red{0})\oplus\O_{\PP^2}(\green{2})\oplus\O_{\PP^2}(\blue{4})$. We construct a large family of based maps of $\PP(\E)$ with base map being the $d$\textsuperscript{th} toric Frobenius (so $q = 2$) and relative degree $d = 2$. Set
\begin{align*}
    f_3 &= c_3 z_3^2 \\
    f_2 &= a(t_0,t_1,t_2) z_3^2 + b(t_0,t_1,t_2) z_2 z_3 + c_2 z_2^2 \\
    f_1 &= c(t_0,t_1,t_2) z_3^2 + d(t_0,t_1,t_2) z_2^2 + e(t_0,t_1,t_2) z_2 z_3 + f(t_0,t_1,t_2) z_1 z_3 + g(t_0,t_1,t_2) z_1 z_2 + c_1 z_1^2
\end{align*}
where (for clarity, we have used colors to identify line bundle degrees)
\begin{align*}
    a & \in \Gamma \bigl(\PP^2, \O_{\PP^2}(\blue{4} \cdot 2 - 2q) \bigr) = \Gamma \bigl( \PP^2, \O_{\PP^2}(4) \bigr), && b \in \Gamma \bigl(\PP^2, \O_{\PP^2}(\blue{4} \cdot 1 + \green{2} \cdot 1 - 2q) \bigr) = \Gamma \bigl( \PP^2, \O_{\PP^2}(2) \bigr), \\ 
    c & \in \Gamma \bigl(\PP^2, \O_{\PP^2}(\blue{4} \cdot 2 - 0q) \bigr) = \Gamma \bigl( \PP^2, \O_{\PP^2}(8) \bigr), && d \in \Gamma \bigl(\PP^2, \O_{\PP^2}(\green{2} \cdot 2 - 0q) \bigr) = \Gamma \bigl( \PP^2, \O_{\PP^2}(4) \bigr), \\ 
    e & \in \Gamma \bigl(\PP^2, \O_{\PP^2}(\blue{4} \cdot 1 + \green{2} \cdot 1 - 0q) \bigr) = \Gamma \bigl( \PP^2, \O_{\PP^2}(6) \bigr), && f \in \Gamma \bigl(\PP^2, \O_{\PP^2}(\blue{4} \cdot 1 + \red{0} \cdot 1 - 0q) \bigr) = \Gamma \bigl( \PP^2, \O_{\PP^2}(4) \bigr), \\ 
    g & \in \Gamma \bigl(\PP^2, \O_{\PP^2}(\green{2} \cdot 1 + \red{0} \cdot 1 - 0q) \bigr) = \Gamma \bigl( \PP^2, \O_{\PP^2}(2) \bigr), && ~
\end{align*}
and $c_1, c_2, c_3$ are constants. By construction, $f_1, f_2$, and $f_3$ have no common zero due to the lower-triangular structure and assumption that $c_1, c_2, c_3$ are constant. Therefore, by Theorem~\ref{thm:linebundlecase}, this specifies the data of a based map of $\PP(\E)$ with base map being the $d$\textsuperscript{th} toric Frobenius on $\PP^2$ and relative degree $2$. Computing the dimensions of the space of polynomials to choose from we have $45 \cdot 28^2 \cdot 15^3 \cdot 6 = 714, 420, 000$ degrees of freedom in this family.
\end{example}

\subsection{Explicit classifications of morphisms on Hirzeburch surfaces}

Fix an integer $n \geq 1$ and let $\mathbb{F}_n=\PP(\O_{\PP^1}\oplus \O_{\PP^1}(n))$ denote the $n$\textsuperscript{th} Hirzebruch surface. In this subsection we classify all surjective morphisms of $\mathbb{F}_n$.

\begin{lemma}
\label{lem:HirzebruchIsBased}
Let $n>0$ and $\psi\colon \mathbb{F}_n\ra \mathbb{F}_n$ be a surjective morphism. Then, $\psi$ is a based map. 
\end{lemma}

\begin{proof}
Consider $\mathbb{F}_n$ as a projective bundle over $\PP^1$ with projection map $\pi$, and define $\L=\O_{\mathbb{F}_n}(1)$ and $F=\pi^*\O_{\bP^1}(1)$. Recall that the nef cone of $\mathbb{F}_n$ is generated by $\L$ and $F$. As $\psi^*$ preserves the nef cone of $\mathbb{F}_n$, it preserves its boundary rays. Working additively in the Chow ring, we show that $\psi^*\L=d\L$ and $\psi^*F=dF$. 

Towards a contradiction suppose that $\psi^*\L=\alpha F$ for some scalar $\alpha$. Then $\psi^*\L\cdot \psi^*\L=0$, since $F^2=0$. Furthermore, recall that $\L^2 = n[p]$, where $[p]$ is the class of a point. Hence $\psi^*\L \cdot \psi^*\L = \psi^*(\L^2) = \psi^*(n[p]) = n\psi^*([p]) = 0$. However, this cannot be the case as the pullback of the class of a point is a multiple of the class of a point, in particular $\psi^*([p]) = d^\prime [p]$ where $d^\prime$ is the topological degree of $\psi$.

Therefore $\psi^*$ preserves the rays of the nef cone. By Lemma~\ref{lem:pullback-is-constant} we know that $\psi^*$ preserves the rays of the pseudo-effective cone, and that $\L$ is not a boundary ray of the pseudo-effective cone. It follows that $\psi^*$ has three noncollinear eigenvectors, so $\psi^*=d\cdot\operatorname{Id}$ for some positive integer $d$. In particular, $\psi^*\L=d\L$ and $\psi^*F=dF$. 

Now write $\psi_*F=\alpha \L+\beta F$. Then $\psi_*F\cdot F=F\cdot (\alpha\L+\beta F)=\alpha F\cdot\L=\alpha$ as the intersection of $\O_{\mathbb{F}_n}(1)$ with a fibre is 1. On the other hand $\psi_*F\cdot F=F\cdot\psi^*F=F\cdot dF=dF^2=0$. We conclude that $\alpha=0$. In other words $\psi_*F=(\deg \psi\vert_F)[\psi(F)]=\beta [F]$ in the Chow ring. This tells us that $\pi_* (\deg \psi\vert_F)[\psi(F)]=0$, since $\pi_*[F]=0$. Because the pushforward is zero precisely when the dimension of $\pi(\psi(F))$ is strictly smaller than $\dim \psi(F)=1$, we see that $\pi(\psi(F))$ is irreducible and $0$-dimensional, in other words, a point. It follows that $\psi(F)$ is a fibre of $\pi$, so we conclude that $\psi$ maps fibres to fibres. 

Finally let $s\colon \bP^1\ra \mathbb{F}_n$ be a section so that for any point $p \in \PP^1$ and $q \in \pi^{-1}(p)$ we have $(s \circ \pi)(q) \in \pi^{-1}(p)$ (such a section exists by \cite[Exercise~II.7.8]{hartshorne}). Define $\varphi\colon \bP^1\ra \bP^1$ by $\pi \circ \psi \circ s$. Since $\psi$ maps fibres to fibres, for any $p \in \PP^1$, we have that $(\pi \circ \psi)(\pi^{-1}(p)) = \pi( \pi^{-1}(p)) = p$. For any $q \in \mathbb{F}_1$, if we set $p \coloneqq (\pi \circ \psi)(q)$, then $q \in \pi^{-1}(p)$ and so $p^\prime \coloneqq (s \circ \pi)(q) \in \pi^{-1}(p)$. Consequently we get
\begin{align*}
    (\varphi \circ \pi)(q) &= (\pi \circ \psi \circ s \circ \pi)(q) \\
    &= (\pi \circ \psi)(p^\prime) \\
    &= p = (\pi \circ \psi)(q)
\end{align*}
In other words, we have that $\varphi \circ \pi = \pi \circ \psi$, which shows $\psi$ is a based map with base map $\varphi$.
\end{proof}

Theorem~\ref{thm:linebundlecase} then gives us a classification of all surjective endomorphisms of $\mathbb{F}_n$.

\begin{corollary}
Every surjective endomorphism of the $n$\textsuperscript{th} Hirzebruch surface $\mathbb{F}_n = \PP \bigl( \O_{\PP^1} \oplus \O_{\PP^1}(1) \bigr)$ is given by the data of a surjective endomorphism $\varphi \colon \PP^1 \rightarrow \PP^1$ of degree $d$ and pairs
\begin{align*}
    F_1([x:y],z_1,z_2) = \sum_{i=0}^ds_i([x:y])z_1^{d-i}z_2^i, && F_2([x:y],z_1,z_2)=cz_2^d,
\end{align*}
where $c, s_0$ are non-zero constants, and $s_i\in \O_{\PP^1}(ni)$ for all $i=1,\dots, d$. Concretely, a point $(p,[z_1:z_2]) \in \mathbb{F}_n$ is mapped under $\psi$ by the rule
\begin{equation*}
    (p,[z_1:z_2])\mapsto \left(\varphi(p), \left[\sum_{i=0}^ds_i(p)z_1^{d-i}z_2^i:cz_2^d\right]\right).
\end{equation*}
\end{corollary}

\begin{proof}
In the proof of Lemma~\ref{lem:HirzebruchIsBased}, we see that $\varphi$ has degree $d$. Hence, by Theorem~\ref{thm:linebundlecase}, describing a based map $\psi$ with base map $\phi$ and relative degree $d$ is equivalent to, for each partition $\lambda_1 + \lambda_2 = d$, choosing global sections
\begin{align*}
    a_{1,\lambda} \in \Gamma \bigl( \PP^1, \O_{\PP^1}(n\lambda_2) \bigr) && a_{2,\lambda} \in \Gamma \bigl( \PP^1, \O_{\PP^1}(n\lambda_2) \otimes \O_{\PP^2}(-nd) \bigr).
\end{align*}
In the latter case, the only time there are non-zero global sections is when $\lambda_2 = d$, and so we set $c = a_{2,(0,d)}$ which is some constant. In the former case, we have no restrictions on the sections, so we set $s_i \coloneqq a_{1,(d-i,i)}$. When $i = 0$, we see that $s_0$ is some constant. This establishes the form of $F_1$ and $F_2$.

Let us show that both $c$ and $s_0$ must be non-zero. If $s_0 = 0$, then both $F_1$ and $F_2$ would be divisible by $z_2$, hence vanishing at a common point. On the other hand if $c = 0$, then $F_2 = 0$ and so it certainly shares a common zero with $F_1$.
\end{proof}

\section{The tangent and cotangent bundles of a toric variety}
\label{subsec:tangent-and-cotangent}
Throughout this section $X$ is an $n$-dimensional smooth projective toric variety with fan $\Sigma$. Our goal is to prove Theorem~\ref{thm:mainTC}, stating that the projectivized tangent and cotangent bundles $\PP(T_X)\to X$ and $\PP(\Omega_X)\to X$ do not admit any non-automorphic base-toric surjective endomorphisms. We further suppose that $X$ is not isomorphic to $(\PP^1)^n$ as we carried out a classification of the base-toric maps of the projectivized (co)tangent bundle in Example~\ref{ex:PP1-classification}; in particular we may assume $\dim X>1$.

Our approach hinges on the decomposition of the tangent bundle along the toric invariant curves of $X$. We say two maximal cones $\sigma,\sigma'\in\Sigma$ are \emph{adjacent} if their intersection is an $(n-1)$-dimensional cone $\tau$. The orbit closure corresponding to such $\tau$ is an invariant curve $C_\tau\cong\PP^1$, and the restriction of the tangent bundle to $C_\tau$ splits as
\[
    T_X|_{C_\tau} = \mathcal{O}(a_1)\oplus\cdots\oplus\mathcal{O}(a_{n-1})\oplus\mathcal{O}(2),
\]
where the integers $a_1,\dots,a_{n-1}$ arise from the wall relation between $\sigma$ and $\sigma'$; see Subsection~\ref{subsec: transition functions and coordinates}.

Our proof is subdivided into three cases determined by the splitting type of the tangent bundle:

\begin{enumerate}
\item[(\namedlabel{i:overall-case1}{1})] When all $a_i > 0$ for all invariant curves, $T_X$ is ample and $X \cong \PP^n$ by Mori's Theorem. Both results then follow from a theorem of Amerik and Kuznetsova.

\item[(\namedlabel{i:overall-case2}{2})] Similarly, when all $a_i \geq 0$ for all invariant curves, $T_X$ is nef and $X \cong \prod \PP^{n_i}$, and the result follows once again by the work of Amerik and Kuznetsova.

\item[(\namedlabel{i:overall-case3}{3})] When some $a_i < 0$ for at least one invariant curve, a more intricate analysis is required.
\end{enumerate}

The cases \eqref{i:overall-case1} and \eqref{i:overall-case2} are treated simultaneously in Subsection~\ref{subsec: Cases 1 and 2}. After this, the remainder of the section is devoted to the study of case \eqref{i:overall-case3}. This case is more challenging, and our proofs rely on the specific transition functions of each one of these bundles. Concretely, we demonstrate that when a local surjective endomorphism is defined over an affine toric subvariety $U_\sigma\subseteq X$, it cannot be extended to some adjacent affine toric subvariety. While we conjecture our techniques extend beyond the cases presented, we conclude with an example of a smooth projective toric variety where surjective endomorphisms can be glued along two adjacent subvarieties. We hypothesize that such local constructions cannot be glued into a global surjective endomorphism, though establishing this rigorously would require new techniques.

\subsection{Cases \ref{i:overall-case1} and \ref{i:overall-case2}: nef tangent bundle}\label{subsec: Cases 1 and 2}

According to \cite[Theorem~2.1]{hering2010positivity}, a toric vector bundle on a toric variety is ample (resp. nef) if and only if  its restriction to every invariant curve $C_\tau\subseteq X$ is ample (resp. nef). When $T_X$ is ample, Mori proved that $X\cong\PP^n$ \cite{mori1979projective} (see also \cite{Wu2024} for a proof in the toric case). Similarly, using a result of Fujino-Sato \cite{FujinoSato2009}, we prove in Proposition~\ref{sec: tangent nef implies product of projective spaces} that when $T_X$ is nef, $X$ is isomorphic to a product of projective spaces. In both cases the question of existence of surjective endomorphisms has been fully answered by Amerik and Kuznetsova:

\begin{lemma}[{\cite[Theorem~1]{Amerik2}}]
    Let $X$ be a simply-connected projective variety such that for any line bundle $L$ its first cohomology $H^1(X,L) = 0$. Let $\E$ be a vector bundle of rank $r>1$ on $X$. If there exists a surjective endomorphism of $\PP(E)\to X$ of degree greater than one on the fibres, then $E$ splits into a direct sum of line bundles.
\end{lemma}

This result applies directly to our setting as both $\PP^n$ and products of projective spaces satisfy the cohomological vanishing condition. Indeed, the vanishing of $H^1(\PP^n, L)$ for any line bundle $L$ when $n \geq 2$ is well-known. For products of projective spaces the vanishing follows from K\"unneth's formula (see, for example, \cite[Theorem~14]{kempf1980some}):
\[
    H^1 \bigl( \PP^{d_1}\times\cdots\times\PP^{d_k},\mathcal{O}(b_1,\dots,b_k) \bigr) = \bigotimes_{1\leq i\leq k} H^1 \bigl( \PP^{d_i},\mathcal{O}(b_i) \bigr),
\]
which is equal to $0$ provided that at least one $d_i>1$. Since the tangent and cotangent bundles of projective spaces of dimension at least two do not split, it follows that their projectivizations admit no surjective endomorphisms.

Therefore, this gives a proof of Theorem~\ref{thm:mainTC} in the case of $X$ being a product of projective spaces, as $T_X$ being nef implies $X$ is isomorphic to a product of projective spaces.

\begin{proposition}\label{sec: tangent nef implies product of projective spaces}
    Let $X$ be a smooth projective toric variety such that $T_X$ is nef. Then, $X\cong\PP^{d_1}\times\cdots\times\PP^{d_k}$ for some positive integers $d_1,\dots,d_k$.
\end{proposition}

\begin{proof}
    By \cite[Corollary~8.4.4]{Lazarsfeld-PosII}, every effective divisor on a smooth projective variety with nef tangent bundle is nef, that is, $\operatorname{Nef}(X) = \overline{\operatorname{Eff}}(X)$. For $X$ as in the statement, this equality implies that $X$ is a product of projective spaces by \cite[Proposition~5.3]{FujinoSato2009}, as desired. 
\end{proof}

\subsection{The transition function method}

To attack the situation in which $X$ is not a product of projective spaces, we describe surjective endomorphisms explicitly using transition functions. Firstly, however, we make a simplifying reduction that the map on the base $X$ is just a toric Frobenius map.

\begin{proposition}\label{prop:firbres cor}
    Let $X_\Sigma$ be a smooth projective toric variety defined by a fan $\Sigma$. Take $\E$ be isomorphic to either the tangent bundle or cotangent bundle of $X$. Suppose we have a diagram
    \[
        \xymatrix{\bP(\E)\ar[r]^\psi\ar[d]_\pi & \bP(\E)\ar[d]^\pi\\ X\ar[r]_\varphi & X}
    \]
    with $\psi$ and $\varphi$ surjective. Then for some integer $m \geq 1$, the iterate $\varphi^m$ is the $d$-th toric Frobenius map. Moreover, $(\psi^m)^*L\cong L^{\otimes d}$ for all $L\in \Pic(\bP(\E))$. 
\end{proposition}

\begin{proof}
    Suppose that $X_\Sigma$ equivariantly splits as
    \[
        X_\Sigma\cong X_{\Sigma_1}\times\ldots X_{\Sigma_t}
    \]
    and that each $X_{\Sigma_i}$ does not split equivariantly. By Lemma~\ref{lemma:toric reduction to mult by d}, we have that there is some integer $m\geq 1$ such that $\varphi^m=[d_1]\times\ldots \times [d_t]$, where $[d_i]$ represents the toric morphism induced by multiplication by $d_i$ on $\Sigma_i$. Therefore, we may assume that $\varphi=[d_1]\times\ldots\times [d_t]$ in what follows. Let $x_i$ be the torus identity in $X_i$ and let $\iota_i\colon X_{\Sigma_i}\ra X_\Sigma$ be the inclusion $X_{\Sigma_i}\hookrightarrow X_\Sigma$ defined by $x\mapsto (x_1,\ldots , x\ldots, x_t)$, where $x$ appears in the $i$th component. Let $p_i\colon X_\Sigma\ra X_{\Sigma_i}$ denote the canonical projection, and define $\V_i=\iota_i^*\E$. 
    
    If $\E_i$ is the tangent or cotangent bundle of $X_{\Sigma_i}$, then $\E=p_1^*\E_1\oplus\ldots \oplus p_t^*\E_t$ and $\V_i=\iota_i^*p_1^*\E_1\oplus \ldots\oplus\iota_i^*p_t^*\E_i$. Note that $\iota_i\circ p_j$ is the identity map when $i=j$, and if $i\neq j$ it is the constant map $X_{\Sigma_i}\ra X_{\Sigma_j}$ with image the torus identity $x_j\in X_j$. It follows that $\V_i=\E_i\oplus \O_{X_{\Sigma_i}}^{\oplus \ell_i}$ where $\ell_i=\dim X_{\Sigma}-\dim X_{\Sigma_i}$. 
    
    Choose a torus-invariant curve $C_i\subseteq X_i$ for each $1\leq i\leq t$. If $X_i\neq \bP^1$ we may assume that $\E_i\vert_{C_i}$ is not trivial, as otherwise $\E_i$ would be trivial by \cite[Theorem~6.4]{hering2010positivity}. If $X_i=\bP^1$, then $\E_i=\O_{\bP^1}(\pm 2)$. We conclude that $\V_i\vert_{C_i}$ is not projectively trivial. Consider the restriction of $\psi$ to $\bP(\E\vert_{C_i})=\bP(\V_i\vert_{C_i})$. Then, by Proposition~\ref{prop:fibreProp2}, the degree of $\psi$ on the fibres is the first dynamical degree $\lambda_1(\varphi\vert_{C_i})$, call this number $d$. This implies that the degree of $\psi$ on the fibres equals $\lambda_1(\varphi\vert_{C_i})$ for all $i=1,\ldots,t$; in particular, all these dynamical degrees are equal to $d$. Finally, as $\varphi\vert_{C_i}$ is simply the morphism $[d_i]$ on the toric variety $C_i\cong \bP^1$, we have that $\lambda_1(\varphi\vert_{C_i})=d_i=d$. We conclude that $\varphi=[d]$ on $X_{\Sigma}$. 
    
    To finish the proof we show the last statement. Recall that $\Pic(\bP(\E))=\bZ\oplus \pi^*\Pic(X_\Sigma)$, where $\bZ$ is generated by $\O_{\bP(\E)}(1)$. Since $\O_{\bP(\E)}(1)$ restricted to any fibre $F$ of $\pi$ is $\O_{F}(1)$, we have that $\psi^*\O_{\bP(\E)}(1)=\O_{\bP(\E)}(d)\otimes\pi^*B$ for some line bundle $B$ on $X_{\Sigma}$. After identifying $c_1(\E)\in A^1(X_\Sigma)$ with $\det \E$ in $\Pic(X_\Sigma)$, Corollary~\ref{cor:trans factor} gives
    \[
        B^{\otimes r}\cong \varphi^*\det\E\otimes\det\E^{-\otimes d} = \O_{X_{\Sigma}},
    \]
    where we use that $\varphi^*M\cong M^{\otimes d}$ for any $M\in \Pic(X_\Sigma)$. Since a smooth projective toric variety has no torsion in its Picard group, we conclude that $B=\O_{X_{\Sigma}}$. The result follows.     
\end{proof}

Now we concern ourselves with an explicit method for describing surjective endomorphisms of projective bundles using transition matrices. Suppose $X$ is a normal projective toric variety and $\E$ is either the tangent or cotangent bundle with a surjective endomorphism $\psi \colon \E \rightarrow \E$. By \cite[Lemma~6.9]{lesieutresatriano2021ksc}, we may always assume that after some iterate of $\psi$, there is a surjective morphism $\varphi \colon X \rightarrow X$ such that we have the commutative diagram
\begin{equation}\label{diagram: vector bundle map general}
\begin{tikzcd}
\PP(\E) \arrow[r, "\psi"] \arrow[d, "\pi" left] & \PP(\E) \arrow[d, "\pi" right] \\
X \arrow[r, "\varphi"] & X.
\end{tikzcd}
\tag{$\dagger$}
\end{equation}
Since our goal is to show that $\PP(\E)$ does not admit any non-identity endomorphisms, this simplifying assumption is without loss of generality.

\newpage
\subsection{The transition functions of tangent and cotangent bundles of a toric variety}\label{subsec: transition functions and coordinates}

In this subsection we present the transition functions for the tangent bundle and cotangent bundle of a toric variety $X$. This will allow us to recast the compatibility condition from Proposition~\ref{prop:transition function method} for these bundles in the following subsections. A detailed treatment of most results in this section can be found in \cite{Wu2024}, other relevant references include \cite{rocco:2014,hering2010positivity}.

Let $X = X_\Sigma$ be an $n$-dimensional smooth projective toric variety. Let $\sigma,\sigma'\in\Sigma$ be two maximal cones intersecting in an $(n-1)$-dimensional face $\tau$. Consider the primitive vectors $\mathbf{v}_1,\dots,\mathbf{v}_n,\mathbf{v}_n'\in N$ such that $\sigma=\langle \mathbf{v}_1,\dots,\mathbf{v}_n\rangle$ and $\sigma'=\langle \mathbf{v}_1,\dots,\mathbf{v}_{n-1},\mathbf{v}_n'\rangle$, so that $\tau=\langle \mathbf{v}_1,\dots,\mathbf{v}_{n-1}\rangle$. Then, the \emph{wall relation} determined by $\sigma$ and $\sigma'$ is the equation
\[
    a_1\mathbf{v}_1+\cdots a_{n-1}\mathbf{v}_{n-1}+\mathbf{v}_n+\mathbf{v}_n' =0,\quad a_i\in\mathbb{Z}.
\]
The wall relation is unique. 

Let $\mathbf{u}_1,\dots,\mathbf{u}_n$ be the primitive generators of $\sigma^\vee$, and $\mathbf{u}_1',\dots,\mathbf{u}_n'$ be the primitive generators of $(\sigma')^\vee$. Then, dual to the wall relation we have the following equalities:
\begin{align}\label{sec: tangent relation between ui}
    \begin{aligned}
    \mathbf{u}_i &= \mathbf{u}_i' + a_i\mathbf{u}_n,\text{ for all }i=1,\dots,n-1;\\
    \mathbf{u}_n &= -\mathbf{u}_n'.
    \end{aligned}
\end{align}  
%
%These equalities, together with the wall relation, imply that $a_n=2$.
%
Let $U_\sigma$ and $U_{\sigma'}$ be the affine toric varieties corresponding to $\sigma$ and $\sigma'$. The coordinates of $U_\sigma$ are $\chi^{\mathbf{u}_1}, \ldots, \chi^{\mathbf{u}_n}$, while the coordinates of $U_{\sigma'}$ are $\chi^{\mathbf{u}'_1}, \ldots, \chi^{\mathbf{u}'_n}$. According to Equation~\ref{sec: tangent relation between ui}, these coordinates are related along the intersection $U_\sigma\cap U_{\sigma'}$ by:
\[
    \chi^{\mathbf{u}_i} = \chi^{\mathbf{u}_i' - a_i\mathbf{u}_n'}\qquad\text{and}\qquad \chi^{\mathbf{u}_n} = \chi^{-\mathbf{u}_n'}.
\]
For simplicity, for the remainder of this section we use the following notation for the coordinates of $U_{\sigma'}$:
\begin{align*}
    x_i &:= \chi^{\mathbf{u}_i'},\text{ for all }i=1,\dots,n-1,\\
    y\; &:= \chi^{\mathbf{u}_n'}.
\end{align*}

In this notation, the coordinates of $U_\sigma$ are $x_1y^{-a_1}, \ldots, x_{n-1}y^{-a_{n-1}}$ and $y^{-1}$. The transition function of the tangent bundle $T_X$ from $U_{\sigma'}$ to $U_\sigma$ is the Jacobian matrix $J=J_{\sigma\leftarrow\sigma^\prime}$ of this coordinate change, that is,
\[
    J =
    \begin{pmatrix}
        \dfrac{\partial \chi^{\mathbf{u}_i}}{\partial \chi^{\mathbf{u}_j'}}
    \end{pmatrix}_{ij},
\]
and a direct computation shows that
\begin{equation}\label{sec: tangent equation transition tangent}
J = \begin{bmatrix} y^{-a_1} & 0 & \cdots & 0 & -a_1 x_1 y^{-a_1 - 1} \\
                    0 & y^{-a_2} & \cdots & 0 & -a_2 x_2 y^{-a_2 - 1} \\
                    \vdots & ~ & \ddots & \vdots & \vdots \\
                    0 & \cdots & 0 & y^{-a_{n-1}} & -a_{n-1} x_{n-1} y^{-a_{n-1} - 1} \\
                    0 & \cdots & 0 & 0 & -y^{-2}
    \end{bmatrix}.
\end{equation}

In terms of these coordinates the invariant curve $C_\tau$ is defined by the equations $x_1=\cdots=x_{n-1}=0$.

\begin{corollary}
The restriction $T_X|_{C_\tau}$ of the tangent bundle to the invariant curve $C_\tau$ splits into the following direct sum of line bundles
\[
    T_X\vert_{C_\tau}=\O_{\PP^1}(a_1)\oplus\cdots\oplus\O_{\PP^1}(a_{n-1})\oplus\O_{\PP^1}(2).
\]
\end{corollary}

Similarly, the transition function $J^{\dag}$ of the cotangent bundle of $X$ is the inverse transpose of $J$. A direct computation shows that
\begin{equation}\label{sec: tangent equation transition COtangent}
    J^{\dag} := (J^\text{\sffamily T})^{-1} =
    \begin{bmatrix}
    y^{a_1} & 0           & \cdots & 0 & 0\\
    0       & y^{a_2}     & \cdots & 0 & 0\\
    \vdots       & \vdots & \ddots  & \vdots & \vdots \\                 
    0       &   0          & \cdots & y^{a_{n-1}} & 0 \\
    -a_1x_1y & -a_2x_2y& \cdots & -a_{n-1}x_{n-1}y & -y^{2}
    \end{bmatrix}.
\end{equation}

\subsection{Case~\ref{i:overall-case3}: using the transition functions}

Let $X=X_\Sigma$ be a smooth projective $n$-dimensional toric variety. Consider two adjacent maximal cones $\sigma,\sigma'\in\Sigma$, and let $x_1,\dots,x_{n-1},y$ be the coordinates of the affine toric subvariety $U_{\sigma'}\subseteq X$, as defined in Subsection~\ref{subsec: transition functions and coordinates}. In this subsection we translate the commutativity of the diagram~\eqref{eq:dagger3} in Proposition~\ref{prop:transition function method} into a linear system of polynomials. We then use this framework to prove Theorem~\ref{thm:mainTC} for the tangent bundle.

An $n$-tuple of nonnegative integers $\lambda=(\lambda_1,\dots,\lambda_n)$ is said to be a \emph{composition} of a positive integer $d$ of length $n$ if $\lambda_1+\cdots+\lambda_n=d$, this is written as $\lambda\vDash d$ when there is no ambiguity in the number of parts. Given a composition $\lambda=(\lambda_1,\dots,\lambda_n)$ of $d$, define $\lambda^\circ$ to be the $(n-1)$-tuple obtained by removing its last coordinate, that is, $\lambda^\circ = (\lambda_1,\dots,\lambda_{n-1})$. 

Given a set of variables $z_1,\dots,z_n$ and an $n$-tuple of integers $\lambda$, define
\[
    \mathbf{z}^\lambda = z_1^{\lambda_1}z_2^{\lambda_2}\cdots z_n^{\lambda_n}
\]
and similarly
\[
    \mathbf{z}^{\lambda^\circ} = z_1^{\lambda_1}z_2^{\lambda_2}\cdots z_{n-1}^{\lambda_{n-1}}.
\]
 The set of all compositions of $d$ of length $n$ admits a poset structure defined via $\mu\leq\lambda$ if $\mu_i\leq\lambda_i$ for all $i=1,\dots,n-1$. Note that if $\mu\leq\lambda$, then $\mu_n\geq\lambda_n$.

\begin{lemma}[Compatibility conditions for tangent bundle]\label{lem: tangent compatibility condition tangent bundle}
    Let $J=J_{\sigma\leftarrow\sigma'}$ be the transition matrix of $T_X$ from $U_{\sigma'}$ to $U_{\sigma}$. Then, the compatibility conditions in Proposition~\ref{prop:transition function method} are equivalent to the following equalities:
    \begin{align}
        \tag{$\textnormal{C}_j$} g_j &= y^{da_j} (\Sym^d J)(f_j) \text{ for all } j = 1, 2, \ldots, n-1, \label{eqn:comp-Cj} \\
        \tag{$\textnormal{C}_n$} g_n &= -y^{2d} (\Sym^d J)(f_n) - \sum_{i=1}^{n-1} a_i x_i^d y^{d} (\Sym^d J)(f_i). \label{eqn:comp-Cn}
    \end{align}
    Moreover, for any $f(z_1,\dots,z_n)\in\mathcal{O}_X(U_{\sigma'})[z_1,\dots,z_n]$ of degree $d$ we have that
    \begin{align}\label{eqn:zlambda-of-symd}
    (\Sym^d J) (f) = \sum_{\lambda\vDash d}\mathbf{z}^{\lambda}\sum_{\mu\leq\lambda}C_{\mu,\lambda}m_{\mu,\lambda}\left[\mathbf{z}^\mu\right](f),
    \end{align}
    where
    \begin{align}\label{eqn:monomial-formula}
    \begin{aligned}
        m_{\mu, \lambda} &= \mathbf{x}^{\lambda^\circ - \mu^\circ} y^{-\mathbf{a} \cdot \lambda^\circ - \lambda_n - \mu_n},  \\
        C_{\mu, \lambda} &= (-1)^{\mu_n} \binom{\mu _n}{\lambda_1 - \mu_1, \lambda_2 - \mu_2, \ldots, \lambda_{n-1} - \mu_{n-1},\lambda_n} \mathbf{a}^{\lambda^\circ - \mu^\circ},
    \end{aligned} 
    \end{align}
    and we are using the convention that $0^0=1$.
\end{lemma}

Before proceeding to the proof of this result we present an auxiliary lemma:

\begin{lemma}\label{lem: auxiliary coefficients tangent}
    Let $\mu,\lambda$ be two compositions of $d$ of length $n$. Then, 
    \begin{align} \left[\mathbf{z}^\lambda\right]\left((\Sym^d J)(\mathbf{z}^\mu)\right) = \left\{
    \begin{aligned}
        &C_{\mu,\lambda}\,m_{\mu, \lambda},\text{ if }\mu\leq\lambda;\\
        &0\text{ otherwise.} 
    \end{aligned} \right.
    \end{align}
    for $C_{\mu,\lambda}$ and $m_{\mu, \lambda}$ as in Equation~\eqref{eqn:monomial-formula}.
\end{lemma}
\begin{proof}
Using Remark~\ref{rem:compatibility}, we have that
\[
    (\Sym^d J)(\mathbf{z}^\mu) = \left(y^{-a_1}z_1\right)^{\mu_1}\cdots\left(y^{-a_{n-1}}z_{n-1}\right)^{\mu_{n-1}}
    \left(
        -\sum_{i=1}^{n-1} a_ix_iy^{-a_i-1}z_i - y^{-2}z_n
    \right)^{\mu_n}.
\]
If $\mu\not\leq\lambda$, then $\mu_i>\lambda_i$ for some $1\leq i\leq n-1$ and it becomes impossible for the monomial $\mathbf{z}^{\lambda}$ to appear in the expansion of $(\Sym^d J)(\mathbf{z}^\mu)$. Let us then suppose that $\mu\leq\lambda$ and expand the last term in the previous formula: 
\[
    \left(
        -\sum_{i=1}^{n-1} a_ix_iy^{-a_i-1}z_i - y^{-2}z_n
    \right)^{\mu_n} 
    =
    (-1)^{\mu_n}\sum_{\kappa\vDash\mu_n}\binom{\mu_n}{\kappa}y^{-2\kappa_n}\prod_{i=1}^{n-1}(a_ix_iy^{-a_i-1})^{\kappa_i}
    \mathbf{z}^\kappa.
\]
In particular, the coefficient of the monomial $\mathbf{z}^\lambda$ in the previous sum comes from the unique composition $\tilde\kappa\vDash\mu_n$ such that $\mu^\circ+\tilde\kappa^\circ = \lambda^\circ$. This is given by
\[
    \tilde\kappa = (\lambda_1-\mu_1,\dots,\lambda_{n-1}-\mu_{n-1},\lambda_n),
\]
which is indeed a composition of $\mu_n$. It follows that 
\[             \left[\mathbf{z}^\lambda\right]\left((\Sym^d J)(\mathbf{z}^\mu)\right)
    =C_{\mu, \lambda}
    \left(y^{-a_1}\right)^{\mu_1}\cdots\left(y^{-a_{n-1}}\right)^{\mu_{n-1}}
    y^{-2\lambda_n}\prod_{i=1}^{n-1}(x_iy^{-a_i-1})^{\lambda_i-\mu_i},
\]
where $C_{\mu, \lambda}$ is as in the statement of the lemma. Similarly, the exponent of $x_i$ in this product is $\lambda_i-\mu_i$ for all $i=1,\dots,n-1$. On the other hand, using the fact that both $\mu$ and $\lambda$ are compositions of $d$ one obtains that the exponent of $y$ in this product is $-\sum_{i=1}^{n-1}a_i\lambda_i - \lambda_n - \mu_n$. The result follows.
\end{proof}

\begin{remark}
    The coefficient $C_{\mu,\lambda}$ is zero if and only if there is some $i=1,\dots,n-1$ such that $a_i=0$ and $\lambda_i-\mu_i> 0$. Indeed, this is because in this case $\mathbf{a}^{\lambda^\circ-\mu^\circ} =0$.
\end{remark}

\begin{proof}[Proof of Lemma~\ref{lem: tangent compatibility condition tangent bundle}]
    The compatibility conditions are obtained by simply expanding the equality 
    \[
        \begin{pmatrix} g_1 & \cdots & g_n \end{pmatrix} \varphi^*J^ = \begin{pmatrix}\Sym^d (J)(f_1) & \cdots & \Sym^d (J)(f_n)\end{pmatrix},
    \]
    where $\varphi^*J$ is the matrix obtained from $J$ by replacing $x_j$ with $x_j^d$ for all $j=1,\dots,n-1$, and $y$ with $y^d$. Indeed, this equatility yields
    \begin{align*}
        (\Sym^d J)(f_i) &= y^{-da_i}g_i \text{ for all } i = 1, 2, \ldots, n-1, \\
        (\Sym^d J)(f_n) &= -y^{-2d} g_n - \sum_{i=1}^{n-1} a_i x_i^d y^{-da_i - d} g_i,
    \end{align*}
    from which Equations~\eqref{eqn:comp-Cj} and \eqref{eqn:comp-Cn} follow.

    The proof of Equation~\eqref{eqn:monomial-formula} follows directly from Lemma~\ref{lem: auxiliary coefficients tangent} and the fact that
    \begin{align*}
        \begin{aligned}
            (\Sym^d J) (f) &=\sum_{\lambda\vDash d}z^{\lambda}\sum_{\mu\leq\lambda}\left[\mathbf{z}^\lambda\right]\left((\Sym^d J)(\mathbf{z}^\mu)\right)\cdot\left[\mathbf{z}^\mu\right](f) \\
            &= \sum_{\lambda\vDash d}z^{\lambda}\sum_{\mu\leq\lambda}C_{\mu\lambda}m_{\mu,\lambda}\left[\mathbf{z}^\mu\right](f)
        \end{aligned}
    \end{align*}
    This concludes the proof.
\end{proof}

We leverage the independence and homogeneity of these monomials often, so we record this as a lemma.

\begin{lemma}
\label{lemma:indep-monomials-tangent}
Fix $\eta$ to be a composition of $d$ with $n$ parts. Apply the coefficient extraction operator $\bigl[ \mathbf{z}^\eta \bigr]$ to the compatibility conditions \eqref{eqn:comp-Cj} and \eqref{eqn:comp-Cn}. We have
\begin{enumerate}
    \item In \eqref{eqn:comp-Cj}, the monomials $y^{da_j} m_{\mu,\eta}$ are all linearly independent and homogeneous of the same total degree in $x_1,x_2, \ldots, x_{n-1}, y$.
    \item In \eqref{eqn:comp-Cn}, the monomials $y^{2d} m_{\mu,\eta}$ and $x_i^d y^d m_{\mu',\eta}$ are homogeneous of the same total degree, and they are dependent if and only if $\eta = \mu' = d\textbf{e}_i$ and $\mu = d\textbf{e}_n$.
\end{enumerate}
\end{lemma}

\begin{proof}
One sees homogeneity easily by summing the exponents in \eqref{eqn:monomial-formula}. In the first case, the total degree is $-\mathbf{a} \cdot \eta^\circ - \eta_n - d + da_j$ and in the second case the total degree is $-\mathbf{a} \cdot \eta^\circ - \eta_n + d$. To get linear independence in the first case, we simply observe that the monomials $m_{\mu, \eta}$ are all distinct. In the second case, the same observation holds, except it is possible that $y^{2d} m_{\mu, \eta} = x_i^d y^d m_{\mu', \eta}$. This only holds if $m_{\mu,\eta} = x_i^d$, and if $m_{\mu',\eta} = y^d$, which can both happen if $\eta = \mu' = d\textbf{e}_i$ and $\mu = d\textbf{e}_n$.
\end{proof}

For what follows, it is helpful to distinguish which $a_i$ in the wall relations are specifically positive. Without loss of generality, we order them as follows:
\begin{align}\label{eqn:ai-assumptions}
    \text{$a_1 \leq a_2 \leq\ldots\leq a_m \leq 0$ for some $0 \leq m < n$}, && \text{and} && 0< a_{m+1}\leq a_{m+2}\leq \ldots\leq a_{n-1}.
\end{align}

Finally, our proof leverages the following incarnation of B\'ezout's theorem:

\begin{observation}
\label{obs:Bezout}
    Let $f:\PP^{n-1}\to\PP^{n-1}$ be a nontrivial regular map with components $f=[f_1:\dots:f_n]$. Then, for any $0 \leq m < n$, it is impossible for $n-m$ of the components of $f$ to vanish simultaneously along any $m$-dimensional linear subspace of $\PP^{n-1}$.
\end{observation}

In our situation, we show that each of the $n-m$ polynomials $f_{m+1}, \ldots, f_n$ vanish on the $m$-plane defined by $z_{m+1} = z_{m+2} = \cdots = z_{n-1} = 0$ over the fibre of the torus-fixed point on $U_{\sigma'}$. We are left with $m$ polynomials $f_1, f_2 \ldots, f_m$ which, when restricted to this $m$-plane, have a common zero by B\'ezout's Theorem, contradicting the regularity of $f$.

We now have the tools to prove part of Theorem~\ref{thm:mainTC}, saying that the projectivized tangent bundle of a toric variety does not admit a non-automorphic surjective endomorphism. 

\begin{proof}[Proof of Theorem~\ref{thm:mainTC} for $T_X$]
Since we are assuming that $T_X$ is not a product of line bundles, this implies that $X\neq(\mathbb{P}^1)^n$. In particular, there exist maximal cones $\sigma$ and $\sigma'$ in the fan of $X$ such that $a_k \neq 0$ for some index $1 \leq k \leq n-1$. Suppose $\lambda = (\lambda_1, \lambda_2, \ldots, \lambda_m, 0, 0, \ldots, 0, \lambda_n)$ is a composition of $d$ with $n$ parts and fix $0 \leq m < n$. We claim that for each $j = m+1, m+2, \ldots, n-1, n$, the coefficient $\bigl[ \mathbf{z}^\lambda \bigr] f_j$ has constant term zero when viewed as a polynomial in $x_1, x_2, \ldots, x_{n-1}, y$. The theorem follows from this claim, as it implies that each of $f_{m+1}, \ldots, f_n$ vanish on the $m$-plane defined by $z_{m+1} = z_{m+2} = \cdots = z_{n-1} = 0$ over the fibre of the torus fixed point on $U_{\sigma'}$, given by $x_1=\dots=x_{n-1}=y=0$. This contradicts the regularity of $f$ by Observation~\ref{obs:Bezout}.

To prove the claim we show that the constant coefficient of $[\mathbf{z}^\lambda]f_j$ is zero using Lemma~\ref{lemma:indep-monomials-tangent} for four cases:
\begin{multicols}{2}
\begin{enumerate}
    \item[(\namedlabel{i:case1}{1})] $\lambda\neq d\mathbf{e}_n$ and $j=m+1,\dots,n-1$. 
    \item[(\namedlabel{i:case2}{2})] $\lambda=d\mathbf{e}_n$ and $j=m+1,\dots,n-1$.
    \item[(\namedlabel{i:case3}{3})] $\lambda\neq d\mathbf{e}_n$ and $j=n$. 
    \item[(\namedlabel{i:case4}{4})] $\lambda=d\mathbf{e}_n$ and $j=n$.
\end{enumerate}
\end{multicols}
\noindent Set $\eta \coloneqq \lambda + \lambda_n \mathbf{e}_k - \lambda_n \mathbf{e}_n$ where $\mathbf{e}_j$ denotes a standard basis vector. By construction, $\eta$ is also a composition of $d$ and we have $\lambda \leq \eta$. 

\underline{Case \ref{i:case1}}: Apply the coefficient extraction operator $\bigl[ \mathbf{z}^\eta \bigr]$ to the $j$th condition in \eqref{eqn:comp-Cj}. By \eqref{eqn:zlambda-of-symd}, this is
\begin{align*}
    \bigl[ \mathbf{z}^\eta \bigr] g_j = y^{d a_j} \bigl[ \mathbf{z}^\eta \bigr] (\Sym^d J)(f_j) = \sum_{\mu \leq \eta} C_{\mu,\eta} y^{d a_j} m_{\mu,\eta} \bigl[ \mathbf{z}^\mu \bigr] f_j,
\end{align*}
where $g_j\in\CC[U_\sigma][z_1,\dots,z_n]$ and $f_j\in\CC[U_{\sigma'}][z_1,\dots,z_n]$ for all $j=1,\dots,n$. Let us now consider the constant term of $\bigl[ \mathbf{z}^\lambda \bigr] f_j$, which corresponds to the term $C_{\lambda,\eta} y^{d a_j} m_{\lambda,\eta}$. We claim that $C_{\lambda, \eta} y^{d a_j} m_{\lambda, \eta} \not\in \CC[U_\sigma]$ if $\lambda \neq d \mathbf{e}_n$. In this case Lemma~\eqref{lem: tangent compatibility condition tangent bundle} gives us that $C_{\lambda, \eta} = (-a_k)^{\lambda_n} \neq 0$ and we have
\begin{align*}
y^{da_j} m_{\lambda, \eta} &= \mathbf{x}^{\eta^\circ - \lambda^\circ} y^{-\mathbf{a} \cdot \eta^\circ - \eta_n - \lambda_n + da_j}, && ~ \\
&= x_k^{\lambda_n} y^{-\mathbf{a} \cdot \eta^\circ - \eta_n - \lambda_n + da_j}, && \text{(resolve $\mathbf{x}$)} \\
&= \bigl( x_k^{\lambda_n} y^{-a_k \lambda_n} \bigr) y^{-\mathbf{a} \cdot \lambda^\circ - \lambda_n + da_j}, && \text{(substitute $\eta$)} \\
&= \bigl( x_k y^{-a_k} \bigr)^{\lambda_n} y^{-\mathbf{a} \cdot \lambda^\circ - \lambda_n + da_j}. && \text{(rearranging)} 
\end{align*}
This monomial lies in the ring $\CC[U_\sigma] = \CC \bigl[x_1 y^{-a_1}, x_2 y^{-a_2}, \ldots, x_{n-1} y^{-a_1}, y^{-1} \bigr]$ if and only if the exponent $-\mathbf{a} \cdot \lambda^\circ - \lambda_n + da_j \leq 0$. In particular, if $\lambda \neq d \mathbf{e}_n$, then this inequality is never satisfied. Since this monomial is appearing in the expansion of $\bigl[ \mathbf{z}^\eta \bigr] g_j \in \CC[U_\sigma]$, it must have coefficient zero or be canceled by other terms. However, Lemma~\ref{lemma:indep-monomials-tangent}~(1) tells us that the monomials $y^{da_j} m_{\lambda, \eta}$ are linearly independent and homogeneous. Hence, if $\bigl[ \mathbf{z}^\lambda \bigr] f_j$ had a nonzero constant term, independence would prevent us from canceling $C_{\lambda, \eta} m_{\lambda, \eta} \bigl[ \mathbf{z}^\lambda \bigr] f_j$ via other terms of the same degree, and homogeneity would prevent us from canceling via higher degree terms. We conclude that $\bigl[ \mathbf{z}^\lambda \bigr] f_j$ must have no constant term for all $j=m+1,\dots,n-1$ and $\lambda \neq d \mathbf{e}_n$.

\underline{Case \ref{i:case2}}: Apply the coefficient extraction operator $\bigl[ \mathbf{z}^\eta \bigl]$ to \eqref{eqn:comp-Cn}. We get
\begin{align*}
    \bigl[ \mathbf{z}^\eta \bigl] g_n &= -y^{2d} \bigl[ \mathbf{z}^\eta \bigl] (\Sym^d J)(f_n) - \sum_{i=1}^{n-1} a_i x_i^d y^d \bigl[ \mathbf{z}^\eta \bigl] (\Sym^d J)(f_i), \\
    &= -\sum_{\mu^n \leq \eta} C_{\mu^n, \eta} y^{2d} m_{\mu^n, \eta} \bigl[ \mathbf{z}^{\mu^n} \bigr] f_n - \sum_{i=1}^{n-1} \sum_{\mu^i \leq \eta} C_{\mu^i,\eta} x_i^d y^d m_{\mu^i, \eta} \bigl[ \mathbf{z}^{\mu^i} \bigr] f_i.
\end{align*}
Let $\lambda = d\textbf{e}_n$ and consider the term $\bigl[ \mathbf{z}^\lambda \bigr] f_j$ which appears in the second sum with coefficient $C_{\lambda,\eta} x_j^d y^d m_{\lambda, \eta}$, where $C_{\lambda,\eta} = (-a_k)^{\lambda_n} = (-a_k)^d\neq 0$ from \eqref{eqn:monomial-formula}. Rewriting $x_j^d = (x_j^d y^{-da_j}) y^{da_j}$, we see that
\begin{align*}
    x_j^d y^d m_{\lambda, \eta} &= (x_j^d y^{-da_j}) \mathbf{x}^{\eta^\circ - (d\mathbf{e}_n)^\circ} y^{-\mathbf{a} \cdot \eta^\circ - \eta_n - \lambda_n + d + da_j} \\
    &= (x_j^d y^{-da_j}) (\mathbf{x}^{\eta^\circ} y^{-\mathbf{a} \cdot \eta^\circ}) y^{da_j}.
\end{align*}
This never lies in the ring $\CC[U_\sigma]$ since $da_j > 0$ and hence, by Lemma~\ref{lemma:indep-monomials-tangent}~(2), we must again conclude that $\bigl[ \mathbf{z}^\lambda \bigr] f_j$ has no constant term, which concludes the proof of the claim for $j = m+1,m+2, \ldots, n-1$.

\underline{Case \ref{i:case3}}: Assume $\lambda \neq d \mathbf{e}_n$ and examine the term $\bigl[ \mathbf{z}^\lambda \bigr] f_n$ in \eqref{eqn:comp-Cn}. This appears with coefficient $C_{\lambda, \eta} y^{2d} m_{\lambda, \eta}$ and once again we see that
\begin{align*}
    y^{2d} m_{\lambda, \eta} &= \mathbf{x}^{\eta^\circ - \lambda^\circ} y^{-\mathbf{a} \cdot \eta^\circ - \eta_n - \lambda_n + 2d} \\
    &= (x_k y^{-a_k})^{\lambda_n} y^{-\mathbf{a} \cdot \lambda^\circ - \lambda_n + 2d}
\end{align*}
does not lie in $\CC[U_\sigma]$ because the exponent in $y$ is positive, as $-\mathbf{a}\cdot\eta^\circ\geq 0$ by our choice of $\lambda$, and $d>\lambda_n$. Using Lemma~\ref{lemma:indep-monomials-tangent}~(2), we conclude that $\bigl[ \mathbf{z}^\lambda \bigr] f_n$ has no constant term as long as $\lambda \neq d\mathbf{e}_n$.

\underline{Case \ref{i:case4}}: In this case instead consider $\eta \coloneqq (d-1)\mathbf{e}_n + \mathbf{e}_k$. The coefficient of $\bigl[ \mathbf{z}^\lambda \bigr] f_n$ in \eqref{eqn:comp-Cn} is $C_{\lambda, \eta} y^{2d} m_{\mu, \eta}$, where $C_{\lambda, \eta} = (-1)^{d-1}a_k\neq 0$. In this case, since $\eta \neq d\textbf{e}_i$ for any $i$, we get that $y^{2d} m_{\mu, \eta}$ is linearly independent to all the other terms by Lemma~\ref{lemma:indep-monomials-tangent}~(2), so we just need to check whether it lies in $\CC[U_\sigma]$. Once again we see that
\begin{align*}
    y^{2d} m_{\lambda, \eta} &= \mathbf{x}^{\eta^\circ - \lambda^\circ} y^{-\mathbf{a} \cdot \eta^\circ - \eta_n - \lambda_n + 2d} \\
    &= (x_k y^{-a_k}) y^{-\mathbf{a} \cdot \lambda^\circ - 2\lambda_n + 2d+1},
\end{align*}
which does not lie in $\CC[U_\sigma]$ because the exponent of $y$ is positive, as $-\mathbf{a} \cdot \lambda^\circ=0$ and $-2\lambda_n + 2d+1=1$. We conclude that $\bigl[ \mathbf{z}^\lambda \bigr] f_n$ has no constant term in this case as well, completing the proof.
\end{proof}

\subsection{The cotangent bundle}

In this subsection we give the complete proof of Theorem~\ref{thm:mainTC}, which states that $\PP(\Omega_X)\to X$ does not admit any surjective toric endomorphisms. (By this we mean maybe that the base map is not toric) The strategy of the proof is similar to the one for the projectivized tangent bundle. Let us begin by recalling our running notation. After this we recast the compatibility conditions given in Proposition~\ref{prop:transition function method} in the specific case of the cotangent bundle of $X$.

Recall that the transition matrix from $U_{\sigma'}$ to $U_\sigma$ was seen in Equation~\eqref{sec: tangent equation transition COtangent} to be

\begin{equation}
    J^{\dag} := (J^{T})^{-1} =
    \begin{bmatrix}
    y^{a_1} & 0           & \cdots & 0 & 0\\
    0       & y^{a_2}     & \cdots & 0 & 0\\
    \vdots       & \vdots & \ddots  & \vdots & \vdots \\                 
    0       &   0          & \cdots & y^{a_{n-1}} & 0 \\
    -a_1x_1y & -a_2x_2y& \cdots & -a_{n-1}x_{n-1}y & -y^{2}
    \end{bmatrix}.
\end{equation}

In order to determine whether this bundle admits any surjective endomorphisms, we recast the compatibility conditions from Lemma~REF in this context:

\begin{lemma}[Compatibility conditions for cotangent bundle]\label{lem: cotangent compatibility condition}
    Let $J^\dag=J^\dag_{\sigma\leftarrow\sigma'}$ be the transition matrix of $\Omega_X$ from $U_{\sigma'}$ to $U_{\sigma}$. Then, the compatibility conditions in Proposition~\ref{prop:transition function method} are equivalent to the following equalities:
    \begin{align}
        \tag{$\textnormal{D}_j$} g_j&=y^{-da_j}\left(\Sym^d(J^{\dag})(f_j)+a_jx_j^dy^dg_n\right) \text{ for all } j = 1, 2, \ldots, n-1, \label{eqn:comp-Dj} \\
        \tag{$\textnormal{D}_n$} g_n&=-y^{-2d}\Sym^d(J^{\dag})(f_n). \label{eqn:comp-Dn}
    \end{align}
    Moreover, for any homogeneous $f(z_1,\dots,z_n)\in\CC[U_{\sigma'}][z_1,\dots,z_n]$ of degree $d$, we have that
    \begin{align}
    \left[z_{1}^{d-1}z_n\right](\Sym^d(J^{\dag})(f))& =-y^{a_1(d-1)+1}\left(y\left[z_{1}^{d-1}z_n\right](f) + da_{1}x_{1}\left[z_{1}^d\right](f)\right) \label{eq:coefextract1}\\
    \left[z_n^d\right](\Sym^d(J^{\dag})(f))& =(-1)^dy^d\sum_{\lambda\vDash d}\mathbf{a}^{\lambda^\circ}\mathbf{x}^{\lambda^\circ}y^{\lambda_n} [\mathbf{z}^\lambda](f),\label{eq:coefextract2}
    \end{align}
    and we are using the convention that $0^0=1$.
\end{lemma}

Note that in this case we are only interested in two monomials of $(\Sym^d(J^{\dag})(f))$. 

\begin{proof}
    The compatibility conditions are obtained by simply expanding the equality 
    \[
        \begin{pmatrix} g_1 & \cdots & g_n \end{pmatrix} \varphi^*J^\dag = \begin{pmatrix}\Sym^d (J^\dag)(f_1) & \cdots & \Sym^d (J^\dag)(f_n)\end{pmatrix},
    \]
    where $\varphi^*J^\dag$ is the matrix obtained from $J^\dag$ by replacing $x_j$ with $x_j^d$ for all $j=1,\dots,n-1$, and $y$ with $y^d$.
    
    In order to compute the two coefficients we are interested in $(\Sym^d(J^{\dag})(f))$ note that
    
    \begin{align*}
        \Sym^d(J^{\dag})(f) &= f\left( J^{-1}\begin{pmatrix} z1 &\cdots & z_n\end{pmatrix}^T\right) \\
        & = f(y^{a_1}z_1 - a_1x_1y z_n, \dots, y^{a_1{n-1}}z_{n-1} - a_{n-1}x_{n-1}y z_n, -y^2z_n) \\
        & = \sum_{\lambda\vDash d} (y^{a_1}z_1 - a_1x_1y z_n)^{\lambda_1}\cdots (y^{a_{n-1}}z_{n-1} - a_{n-1}x_{n-1}y z_n)^{\lambda_{n-1}} (-y^2z_n)^{\lambda_n} [\mathbf{z}^\lambda](f).
    \end{align*}
    The only terms contributing to the coefficient of $z_{1}^{d-1}$ are those corresponding to the compositions $(d,0,\dots,0)$ and $(d-1,0,\dots,0,1)$. It follows that 
    \begin{align*}
        \left[z_1^{d-1}z_n\right](\Sym^d(J^{\dag})(f)) &= -y^{a_1(d-1)+2}[\mathbf{z}^{(d-1,0,\dots,0,1)}](f) + \binom{d}{2}y^{a_1(d-1)}(-a_{1}x_{1}y)[\mathbf{z}^{(d,0,\dots,0)}](f)\\
        &= -y^{a_1(d-1)+1}\left(y\left[z_{1}^{d-1}z_n\right](f) + da_{1}x_{1}\left[z_{1}^d\right](f)\right).
    \end{align*}
    On the other hand, every composition contributes to the coefficient $\left[z_n^d\right](\Sym^d(J^{\dag})(f))$. More precisely, their contribution arises from the terms in which the $z_n$ term of each binomial is exponentiated fully. It follows that 
    \begin{align*}
        \left[z_n^d\right](\Sym^d(J^{\dag})(f)) &= (-1)^d\sum_{\lambda\vDash d} (a_1x_1y)^{\lambda_1}\cdots (a_{n-1}x_{n-1}y)^{\lambda_{n-1}}(y^2)^{\lambda_n} [\mathbf{z}^\lambda](f) \\
        &= (-1)^dy^d\sum_{\lambda\vDash d} \mathbf{a}^{\lambda^\circ}\mathbf{x}^{\lambda^\circ}y^{\lambda_n} [\mathbf{z}^\lambda](f).
    \end{align*}
    This concludes the proof.
\end{proof}

With this we are ready to prove Theorem~\ref{thm:mainTC}, stating that the projectivized cotangent bundle does not admit any surjective endomorphisms.

For simplicity, we use assumptions \eqref{eqn:ai-assumptions}. That is, we suppose without loss of generality that the coefficients of the wall relation between $\sigma$ and $\sigma'$ are ordered in the following way:
\begin{align*}
    \text{$a_1 \leq a_2 \leq \ldots,\leq a_m \leq 0$ for some $0 \leq m < n$}, && \text{and} && 0< a_{m+1}\leq a_{m+2}\leq \ldots\leq a_{n-1}.
\end{align*}

\begin{proof}[Proof of Theorem~\ref{thm:mainTC} for $\Omega_X$]
    The case where $m=0$ and all $0\leq a_i$ was solved in Subsection~\ref{subsec: Cases 1 and 2}, so let us suppose that $a_1<0$, so that $m\geq 1$.

     We claim that all the coefficients of $f_1$, seen as polynomials in $\CC[U_{\sigma'}]=\CC[x_1,\dots,x_{n-1},y]$, vanish identically on the fibre above the torus-fixed point of $U_{\sigma'}$. This contradicts the regularity of $f$, since on this fibre $f_2,\dots,f_n$ have a common zero by B\'ezout's Theorem. The torus-fixed point is defined affine locally by the equations $x_1=\dots=x_{n-1}=y=0$, so this claim is equivalent to showing that $[\mathbf{z}^\lambda]f_1$ has no constant term for all $\lambda\vDash d$. 
     
    Upon substituting the formula for $g_n$ from the compatibility condition \eqref{eqn:comp-Dn} in the compatibility condition \eqref{eqn:comp-Dj} applied to $j=1$, we obtain that
    \[
        g_1 = y^{-da_1}\left(
            \Sym^d(J^\dag)(f_1) - a_1x_1^dy^{-d}\Sym^d(J^\dag)(f_n)
        \right).
    \]
    In particular, using Equation~\eqref{eq:coefextract2}, the $z_n^d$-coefficients in both sides of this equation satisfy
    \begin{align}\label{eqn: cotangent proof cond D1}
    \begin{aligned}
        \left[z_n^d\right] g_1 &= (-1)^dy^{d-da_1}\left(\sum_{\lambda\vDash d}\mathbf{a}^{\lambda^\circ}\mathbf{x}^{\lambda^\circ}y^{\lambda_n} [\mathbf{z}^\lambda](f_1) - a_1x_1^dy^{-d}\sum_{\lambda\vDash d}\mathbf{a}^{\lambda^\circ}\mathbf{x}^{\lambda^\circ}y^{\lambda_n} [\mathbf{z}^\lambda](f_n)\right)%\\
        %&= (-1)^dy^{d-da_1}\sum_{\lambda\vDash d}\mathbf{a}^{\lambda^\circ}\mathbf{x}^{\lambda^\circ}y^{\lambda_n}\left( [\mathbf{z}^\lambda](f_1) - a_1x_1^dy^{-d}[\mathbf{z}^\lambda](f_n)\right).
    \end{aligned}
    \end{align}
    First note that the monomial in Equation~\eqref{eqn: cotangent proof cond D1} corresponding to the constant term of $[\mathbf{z}^\lambda](f_1)$ is $\mathbf{x}^{\lambda^\circ}y^{d-da_1+\lambda_n}$. We claim that this monomial is not an element of $\CC[U_\sigma]$. If it was, then,
    \[
        \mathbf{x}^{\lambda^\circ}y^{d-da_1+\lambda_n}=y^{-\beta}\prod_{i=1}^{n-1}(x_iy^{-a_i})^{\alpha_i}
    \]
    for some $\alpha_1,\dots, \alpha_{n-1},\beta\geq 0$. The $x$-variables force $\alpha_i=\lambda_i$ for all $i=1,\dots,n-1$, but $\beta$ is yet to be determined. Concretely, $\beta$ must satisfy 
    \[
        d-a_{1}d + \lambda_n = -\beta-\sum_{i=1}^{n-1}a_i\lambda_i
    \]
    or, equivalently, 
    \[
        d+\lambda_n-da_{1}\leq -\sum_{i=1}^{n-1}a_i\lambda_i = -\sum_{i=1}^{m}a_i\lambda_i -\sum_{i=m+1}^{n-1}a_i\lambda_i.
    \]
    By hypothesis $a_1<0$ and $a_1\leq a_2\leq \cdots\leq a_m\leq 0$, so the previous inequality is equivalent to
    \[
        d+\lambda_n-da_{1}
        \leq -\sum_{i=1}^{m}a_i\lambda_i -\sum_{i=m+1}^{n-1}a_i\lambda_i 
        \leq -\sum_{i=1}^{m}a_1\lambda_i -\sum_{i=m+1}^{n-1}a_i\lambda_i
        \leq -da_1 -\sum_{i=m+1}^{n-1}a_i\lambda_i.
    \]
    In particular, the last inequality implies that 
    \[
        d+\lambda_n\leq -\sum_{i=m+1}^{n-1}a_i\lambda_i.
    \]
    However, $a_i> 0$ for all $i=m+1,\dots,n-1$, so the system has no solution.

    Therefore, by Equation~\eqref{eqn: cotangent proof cond D1}, the only way in which $[\mathbf{z}^\lambda](f_1)$ could have a constant term is if it is canceled by one of the terms coming from $f_n$. Concretely, this cancellation can occur if and only if there is a nontrivial linear relation of the form
    \begin{equation}\label{eqn: linear relation cotagentn}
        A\mathbf{x}^{\lambda^\circ}y^{\lambda_n}+B
        x_1^d\mathbf{x}^{\mu^{\circ}}y^{\mu_n-d} = 0,
    \end{equation}
    for some $\mu\vDash d$ and a nonzero scalars $A,B\in\CC$. 

    The proof considers two cases:
    \begin{multicols}{2}    
    \begin{enumerate}
    \item[(\namedlabel{i:case1-cot}{1})] $\lambda\neq d\mathbf{e}_1$.
    \item[(\namedlabel{i:case2-cot}{2})] $\lambda=d\mathbf{e}_1$.
    \end{enumerate}
    \end{multicols}

    \underline{Case \ref{i:case1-cot}}: Consider any composition $\lambda\neq d\mathbf{e}_1$. Then, it is impossible for Equation~\eqref{eqn: linear relation cotagentn} to hold, since the the power of $x_1$ in the first term of the left hand side is strictly smaller than $d$. It follows that the constant term of $[\mathbf{z}^\lambda](f_1)$ is zero for all $\lambda\neq d\mathbf{e}_1$. In other words, the only coefficient of $f_1$ that could have a nonzero constant term is $[z_1^d](f_1)$.

    \underline{Case \ref{i:case2-cot}}: Let us suppose that $\lambda= d\mathbf{e}_1$. In order to show that $[z_1^d](f_1)$ has no constant term we consider the $z_1^{d-1}z_n$-coefficient in the compatibility condition \eqref{eqn:comp-Dj} for $j=1$. Then, using Equation~\eqref{eq:coefextract1} one obtains that
    \begin{align*}
        [z_1^{d-1}z_n](g_1) &= y^{-da_1}\left([z_1^{d-1}z_n](\Sym^d(J^{\dag})(f_1))-a_1x_1^dy^{-d}[z_1^{d-1}z_n](\Sym^d(J^{\dag}(f_n))\right) \\
        &= y^{-da_1}\left(-y^{a_1(d-1)+1}\left(y\left[z_{1}^{d-1}z_n\right](f_1) + da_{1}x_{1}\left[z_{1}^d\right](f_1)\right)\right. \\
        &\qquad\qquad\qquad
        \left. +a_1x_1^dy^{da_1-a_1+1-d}\left(y\left[z_{1}^{d-1}z_n\right](f_n) + da_{1}x_{1}\left[z_{1}^d\right](f_n)\right)
        \right)\\
        &= -y^{-a_1+1}\left(y\left[z_{1}^{d-1}z_n\right](f_1) + da_{1}x_{1}\left[z_{1}^d\right](f_1)\right) \\
        &\qquad\qquad\qquad +a_1x_1^dy^{-a_1+1-d}\left(y\left[z_{1}^{d-1}z_n\right](f_n) + da_{1}x_{1}\left[z_{1}^d\right](f_n)\right)
    \end{align*}
    In particular, note that the monomial corresponding to the constant term of $\left[z_{1}^d\right](f_1)$ is $x_{1}y^{-a_1+1}$. Notice that, since $[z_1^{d-1}z_n](g_1)\in\CC[x_1y^{-a_1},\dots,x_{n-1}y^{-a_{n-1}},y^{-1}]$, it is impossible for this monomial to appear as a monomial in $g_1$. Hence, the monomial coming from the constant term of $\left[z_{1}^d\right](f_1)$ must be canceled by a term in the right hand side of the equation. The cancellation could only be done by a monomial coming from one of the remaining three polynomials in the right hand side. These are
    \[
        y^{-a_1+2}\left[z_{1}^{d-1}z_n\right](f_1),\quad 
        x_1^dy^{-a_1+2-d}\left[z_{1}^{d-1}z_n\right](f_n),\quad\text{and}\quad 
        x_{1}^{d+1}y^{-a_1+1-d}\left[z_{1}^d\right](f_n).
    \]
    However, since $f_1,f_n\in\CC[x_1,\dots,x_{n-1},y][\mathbf{z}]$, it is impossible for the monomial $x_{1}y^{-a_1+1}$ to be canceled by any of the three polynomials in the previous display. Indeed, the exponent of $y$ is too large in the first one because $a_1<0$, while the exponent of $x_1$ is too large in the latter two. We then conclude that the constant term of $[z_1^d](f_1)$ is zero.

    Therefore, we have proved that every term of $f_1$ has no constant term when viewed as a polynomial in $\CC[x_1,\dots,x_{n-1},y]$. As explained at the beginning of the proof, this implies that $f_1,\dots,f_n$ have a common zero on the fibre over the torus-invariant point of $U_{\sigma'}$, contradicting the regularity of $f$.
\end{proof}

% Putting this here because we haven't started and I have time to write. Let $\varphi \colon X \rightarrow X$ be a surjective toric morphism. Assume $X$ is complete for simplicity (mainly just need one full dimensional cone...?). Every toric morphism corresponds to a cone-preserving $\mathbb{Z}$-module map $\overline{\varphi} \colon N \rightarrow N$, which is surjective iff $\varphi$ is. In particular, it must send rays to rays. Since there are finitely many rays in the fan of $X$, we must have that $\overline{\varphi}(\rho_i) = c_i \rho_i$ for some $c_i \in \mathbb{Z}$.

% Set $H_c \coloneqq \operatorname{Span} \{\rho_i \colon c_i = c\}$ and consider the collection $\{H_c \neq 0\}_{c \in \mathbb{Z}}$. This is a finite collection of linear subspaces.

% Fact 1: We have $\sum_{c \in \mathbb{Z}} \dim H_c = \dim X$. I don't know a great proof off the top of my head, but this must be true.

% Fact 2: If $c \neq c'$, then for any $v \in H_c$ and $v' \in H_{c'}$, we have $\langle v, v' \rangle = 0$.
% \begin{proof}
% Suppose not. Since $X$ is complete, there must be a full-dimensional maximal cone $\sigma$ which involves two rays $\rho \in H_c$ and $\rho' \in H_{c'}$ by Fact 1. Then $\overline{\varphi}(\sigma) \neq \sigma$, so it is not cone-preserving. Contradiction.
% \end{proof}

\appendix
\section{Proof of Theorem~\ref{thm:chernrelations}}
\label{appendix}
In this Appendix we prove Theorem~\ref{thm:chernrelations}. As a reminder for the reader, we are assuming $X$ is a smooth projective variety, $\E$ is a rank $r$ vector bundle on $X$, and $\psi$ is a based map of $\PP(\E)$ with base map $\varphi$ satisfying $\psi^* = d\L + \alpha$, where $\L = \O_{\PP(\E}(1)$ and $\alpha \in \Pic(X)$. We recall the statement of the theorem.

\begin{theorem}[{Theorem~\ref{thm:chernrelations}}]
Suppose that $\psi$ is a based map with base map $\varphi$ and relative degree $d$. If $\varphi^* c_i(\E) = q^i c_i(\E)$ for some integer $q \neq d$ and all $i \in \mathbb{N}$, then for any $k \geq 1$ we have
\begin{equation*}
    c_k(\E) = \frac{\binom{r}{k} c_1(\E)^k}{r^k}.
\end{equation*}
\end{theorem}

The proof strategy is simple: we inductively use Equation~\ref{eqn:chernclass-conditions} which tells us
\begin{equation}
\label{eq:appendix}
    \sum_{i=0}^r (-1)^i (d \L + \alpha)^{r-i} \varphi^* c_i(\E) = d^r \sum_{i=0}^r (-1)^i \L^{r-i} c_i(\E).
\end{equation}
The meat of the argument is verifying a binomial identity that appears from the inductive assumption.
 
\begin{proof}[Proof of Theorem~\ref{thm:chernrelations}]
We induct on $k$. The base case $k = 1$ is clear. For the inductive step, assume that for $1 \leq i \leq k$ we have 
\begin{equation*}
    c_i(\E) = \frac{\binom{r}{i} c_1(\E)}{r^i}.
\end{equation*}
Using Corollary~\ref{cor:trans factor} and the assumption that $\varphi^* c_i(\E) = q^i c_i(\E)$, we have that $\alpha = \frac{(q-d) c_1(\E)}{r}$ and so we may rewrite Equation~\eqref{eq:appendix} as
\begin{equation}
\label{eq:appendix-step1}
\sum_{i=0}^r (-1)^i \left( d \L + \frac{(q-d)c_1(\E)}{r} \right)^{r-i} q^i c_i(\E) = d^r \sum_{i=0}^r (-1)^i \L^{r-i} c_i(\E).
\end{equation}
Consider the coefficient of $\L^{r-(k+1)}$ on both sides of Equation~\ref{eq:appendix-step1}. We obtain the equality
\begin{equation*}
\sum_{i=0}^{k+1} (-1)^i \binom{r-i}{r-(k+1)} d^{r-(k+1)} \frac{(q-d)^{k+1-i}c_1(\E)^{k+1-i}}{r^{k+1-i}} q^i c_i(\E) = (-1)^{k+1}d^r c_{k+1}(\E).
\end{equation*}
Dividing by $d^{r-(k+1)}$ and moving the $i=k+1$ term to the right-hand side gives
\begin{equation*}
\sum_{i=0}^{k} (-1)^i \binom{r-i}{r-(k+1)} \frac{(q-d)^{k+1-i}c_1(\E)^{k+1-i}}{r^{k+1-i}} q^i c_i(\E) = (-1)^{k+1}(d^{k+1} - q^{k+1}) c_{k+1}(\E).
\end{equation*}
Substituting using the inductive hypothesis to get
\begin{equation*}
\sum_{i=0}^{k} (-1)^i \binom{r-i}{r-(k+1)} \binom{r}{i} \frac{(q-d)^{k+1-i} c_1(\E)^{k+1}}{r^{k+1}} q^i = (-1)^{k+1}(d^{k+1} - q^{k+1}) c_{k+1}(\E).
\end{equation*}
Now use the trinomial revision identity $\binom{r-i}{r-m} \binom{r}{i} = \binom{r}{m} \binom{m}{i}$ for $m = k+1$ to rewrite and get
\begin{align*}
\sum_{i=0}^{k} (-1)^i \binom{r}{k+1} \binom{k+1}{i} \frac{(q-d)^{k+1-i} c_1(\E)^{k+1}}{r^{k+1}} q^i &= (-1)^{k+1}(d^{k+1} - q^{k+1}) c_{k+1}(\E) \\
\frac{\binom{r}{k+1} c_1(\E)^{k+1}}{r^{k+1}} \sum_{i=0}^{k} (-1)^i \binom{k+1}{i} q^i (q-d)^{k+1-i} &= (-1)^{k+1}(d^{k+1} - q^{k+1}) c_{k+1}(\E).
\end{align*}
Therefore it remains to show that
\begin{equation}
\label{eq:appendix-identity}
\sum_{i=0}^{k} (-1)^i \binom{k+1}{i} d^{r-(k+1)} q^i (q-d)^{k+1-i} = (-1)^{k+1}(d^{k+1} - q^{k+1}),
\end{equation}
because, since $q \neq d$, we could divide out by $(-1)^{k+1}(d^{k+1} - q^{k+1})$ and conclude the desired equality. Finally, to see that Equation~\ref{eq:appendix-identity} holds, we observe that
\begin{align*}
    (-1)^{k+1} d^{k+1} = \bigl( (q-d) - q \bigr)^{k+1} &= \sum_{i=0}^{k+1} (-1)^i \binom{k+1}{i} (q-d)^{k+1-i} q^i \\
    &= (-1)^{k+1} q^{k+1} + \sum_{i=0}^{k} (-1)^i \binom{k+1}{i} (q-d)^{k+1-i} q^i.
\end{align*}
Rearranging gives the desired equality and establishes the induction.
\end{proof}

\bibliographystyle{plain}
\bibliography{bib}
\end{document}